\pdfoutput=1
\documentclass{article}
\usepackage{mathrsfs}
\usepackage[nointegrals]{wasysym}
\usepackage[makeroom]{cancel}
\usepackage{graphicx}
\usepackage{comment}
\usepackage{amsfonts}
\usepackage{ucs}
\usepackage{amsthm}
\usepackage{amssymb}
\usepackage[colorlinks=true,linkcolor=blue,citecolor=green]{hyperref}
\usepackage[margin=3cm]{geometry}
\usepackage{stmaryrd}
\usepackage{bussproofs}
\usepackage{mathtools}
\usepackage[all]{xy}
\usepackage{forest}
\forestset{smullyan tableaux/.style={for tree={math content},where n children=1{!1.before computing xy={l=\baselineskip},!1.no edge}{},closed/.style={label=below:$\times$},},}
\usepackage[all]{xy}
\usepackage{fancyhdr}
\usepackage{authblk}
\usepackage{soul}
\usepackage{tikz}
\usetikzlibrary{arrows,calc,patterns,positioning,shapes}
\usetikzlibrary{decorations.pathmorphing}
\tikzset{
modal/.style={>=stealth',shorten >=1pt,shorten <=1pt,auto,
node distance=1.5cm,semithick},
world/.style={circle,draw,minimum size=1cm},
point/.style={circle,draw,fill=black,inner sep=0.5mm},
reflexive/.style={->,in=120,out=60,loop,looseness=#1},
reflexive/.default={5},
reflexive point/.style={->,in=135,out=45,loop,looseness=#1},
reflexive point/.default={25},
}
\newtheorem{theorem}{Theorem}[section]
\newtheorem{proposition}{Proposition}[section]

\theoremstyle{definition}
\newtheorem{definition}{Definition}[section]
\theoremstyle{remark}
\newtheorem{remark}{Remark}[section]

\newtheorem{convention}{Convention}[section]

\newcommand{\BD}{\mathsf{BD}}
\newcommand{\FDE}{\mathbf{FDE}}

\newcommand{\KFDE}{\mathbf{K}_\FDE}
\newcommand{\KBD}{\BD^\Box}
\newcommand{\ignoBD}{\BD^\mathbf{I}}

\newcommand{\knowBD}{\BD^\blacksquare}

\newcommand{\true}{\mathbf{T}}
\newcommand{\both}{\mathbf{B}}
\newcommand{\neither}{\mathbf{N}}
\newcommand{\false}{\mathbf{F}}
\newcommand{\Var}{\mathsf{Var}}
\newcommand{\Lall}{\mathscr{L}_{\Box,\blacksquare,\mathbf{I}}}
\newcommand{\LBox}{\mathscr{L}_\Box}
\newcommand{\Lknow}{\mathscr{L}_\blacksquare}
\newcommand{\Ligno}{\mathscr{L}_\mathbf{I}}
\newcommand{\Tknowigno}{\mathcal{AC}_{\blacksquare,\mathbf{I}}}
\newcommand{\ltrue}{\mathfrak{t}}
\newcommand{\lfalse}{\mathfrak{f}}
\newcommand{\lnontrue}{\overline{\mathfrak{t}}}
\newcommand{\lnonfalse}{\overline{\mathfrak{f}}}

\allowdisplaybreaks
\begin{document}
\providecommand{\keywords}[1]{\small\textbf{Keywords: } #1}
\title{Knowledge and ignorance in Belnap--Dunn logic\thanks{The research of the first author was funded by the grant ANR JCJC 2019, project PRELAP (ANR-19-CE48-0006). The authors wish to thank two anonymous reviewers for their helpful comments and remarks.}}
\author[1]{Daniil Kozhemiachenko}
\affil[1]{INSA Centre Val de Loire, Univ.\ Orl\'{e}ans, LIFO EA 4022, France\\\href{mailto:daniil.kozhemiachenko@insa-cvl.fr}{daniil.kozhemiachenko@insa-cvl.fr} (corresponding author)}
\author[2]{Liubov Vashentseva}
\affil[2]{Department of Logic, Faculty of Philosophy, Lomonosov Moscow State University, Moscow 119991, Russia\\\href{mailto:vashentsevaliubov@gmail.com}{vashentsevaliubov@gmail.com}}
\maketitle
\begin{abstract}
In this paper, we argue that the usual approach to modelling knowledge and belief with the necessity modality $\Box$ does not produce intuitive outcomes in the framework of the Belnap--Dunn logic ($\BD$, alias $\FDE$ --- first-degree entailment). We then motivate and introduce a~non\-standard modality $\blacksquare$ that formalises knowledge and belief in $\BD$ and use $\blacksquare$ to define $\bullet$ and $\blacktriangledown$ that formalise the \emph{unknown truth} and ignorance as \emph{not knowing whether}, respectively. Moreover, we introduce another modality $\mathbf{I}$ that stands for \emph{factive ignorance} and show its connection with $\blacksquare$.

We equip these modalities with Kripke-frame-based semantics and construct a~sound and complete analytic cut system for $\knowBD$ and $\ignoBD$ --- the expansions of $\BD$ with $\blacksquare$ and $\mathbf{I}$. In addition, we show that $\Box$ as it is customarily defined in $\BD$ cannot define any of the introduced modalities, nor, conversely, neither $\blacksquare$ nor $\mathbf{I}$ can define $\Box$. We also demonstrate that $\blacksquare$ and $\mathbf{I}$ are not interdefinable and establish the definability of several important classes of frames using~$\blacksquare$.

\vspace{.5em}

\noindent
\end{abstract}
\section{Introduction\label{sec:introduction}}
Formalising epistemic and doxastic contexts using classical modal logics can produce counter-intuitive outcomes. For example, if $\Box\phi$ is interpreted as “the agent believes in $\phi$”\footnote{It is customary in doxastic and epistemic logics to use $\mathbf{K}$ for the knowledge modality and $\both$ for the belief modality. We do not follow this tradition as the only difference between knowledge and belief modalities that we consider is that they are defined over different classes of frames. The semantics of the modalities themselves is the same and they exhibit the same ‘box-like’ behaviour.}, then $\Box(p\wedge\neg p)\rightarrow\Box q$ is valid in every regular classical modal logic. This means an agent cannot believe in a~contradictory statement without believing in every proposition.

Similarly, if we interpret $\Box$ as a~(classical) knowledge operator, it usually entails the knowability paradox that states that there is no unknown truth. This happens because an instance of \emph{reductio ad absurdum} is used in the derivation (cf., e.g.,~\cite[\S2]{BrogaardSalerno2019}). In addition, since (in normal logics) necessitation is sound, the agents are omniscient because they know every valid formula. In particular, $\Box\neg(p\wedge\neg p)$ is valid which means that the agent is supposed to know that every contradiction is false. However, if we hold that an agent's knowledge must be built upon the facts they have at hand, we might not be inclined to accept the validity of $\Box\neg(p\wedge\neg p)$ because it may happen that the agent has \emph{no information at all} regarding $p$.

These issues pose a~problem not only for the formalisation of knowledge or belief but also for the formalisation of ignorance because it is usually defined as the \emph{lack of knowledge} (“standard view”) or the \emph{lack of true belief} (“new view”; cf.~a~detailed discussion of these two approaches in~\cite{leMorvanPeels2016}). On the other hand, \emph{paraconsistent modal logics} usually do not suffer from the described drawbacks as \emph{reductio ad absurdum} is not valid, and thus can be more intuitive when it comes to the formalisation of belief, knowledge, and ignorance.
\paragraph{Ignorance and not knowing the truth}
The standard view of ignorance has been extensively criticised (cf.,~e.g.,~\cite[\S2.1]{KubyshkinaPetrolo2021} for an overview of the arguments), in particular, because if we assume it (and classical logic), the agents will be ignorant of every (classically) unsatisfiable formula (as indeed, $\neg\Box\phi$ is a~theorem of every modal logic extending $\mathbf{KD}$ when $\phi$ is unsatisfiable). This issue is addressed by assuming that the agent is ignorant of $\phi$ when $\phi$ is true but the agent believes that it is false. A~(classical) logic formalising this treatment of ignorance was proposed in~\cite{KubyshkinaPetrolo2021}.

Another approach to defining ignorance was proposed in~\cite{vanderHoekLomuscio2004} and~\cite{Steinsvold2008}. There the authors interpret “the agent is ignorant about $\phi$” as “the agent does not know whether $\phi$ is true”, i.e., $\neg(\Box\phi\vee\Box\neg\phi)$. Thus, ignorance is treated as contingency modality $\triangledown$ introduced in~\cite{MontgomeryRoutley1966} and then explored in, e.g.,~\cite{Humberstone1995,Zolin1999}.\footnote{Note that it is more customary to treat $\triangle$ (it is \emph{non-contingent} that or \emph{the agent knows whether} the given proposition is true) as the basic operator and define $\triangledown\phi$ as $\neg\triangle\phi$.} An epistemic interpretation of $\triangle$ (knowing whether) was further investigated in~\cite{FanWangvanDitmarsch2015}.

In addition to ignorance, it is worth mentioning the operator “$\phi$ is an unknown truth” ($\phi\wedge\neg\Box\phi$) studied in~\cite{Steinsvold2008}. Note that while the agent does not have a~true belief regarding~$\phi$, the unknown truth does not conform to the new view of ignorance (cf.~\cite{Peels2011} and~\cite[\S2.2]{KubyshkinaPetrolo2021} for a~detailed discussion). In this framework, unknown truth is defined as an accidence operator $\bullet$ introduced in~\cite{Fine1995}\footnote{Again, just as with $\triangledown$, $\bullet\phi$ ($\phi$ is accidental) is treated as a~shorthand for $\neg\circ\phi$ ($\phi$ is not essential).} and further studied in~\cite{Fine2000} and~\cite{Marcos2005essenceaccident}. 
\paragraph{Modal expansions of the Belnap--Dunn logic} Belnap--Dunn logic ($\BD$, alias
First Degree Entailment --- $\FDE$) is a~paraconsistent logic over the $\{\neg,\wedge,\vee\}$ language formulated by Dunn and Belnap in a~series of papers~\cite{Dunn1976,Belnap1977computer,Belnap1977fourvalued}. Semantically, $\BD$ retains the semantical conditions of truth and falsity of $\{\neg,\wedge,\vee\}$-formulas from classical logic but treats them independently (cf.\ Table~\ref{table:BDinformal}).
\begin{table}
\centering
\begin{tabular}{c|c|c}
&\textbf{is true when}&\textbf{is false when}\\\hline
$\neg\phi$&$\phi$ is false&$\phi$ is true\\
$\phi_1\wedge\phi_2$&$\phi_1$ and $\phi_2$ are true&$\phi_1$ is false or $\phi_2$ is false\\
$\phi_1\vee\phi_2$&$\phi_1$ is true or $\phi_2$ is true&$\phi_1$ and $\phi_2$ are false
\end{tabular}
\caption{Truth and falsity conditions of $\BD$-formulas. Note that a formula can also be \emph{both true and false} and \emph{neither true nor false} as truth and falsity conditions are independent.}
\label{table:BDinformal}
\end{table}
Thus, any proposition $\phi$ can have one value (be either exactly true\footnote{Henceforth, when we deal with $\BD$ and its expansions, we reserve the word “true” to mean “at least true”; we use “exactly true” to stand for “true and non-false”. A similar convention is applied for “false” and “exactly false”.} or exactly false), both values (i.e., a~truth-value “glut” --- both true and false) or no value (a~truth-value “gap” --- neither true nor false). It is easy to see from Table~\ref{table:BDinformal} that if all variables of $\phi$ are both true and false, then $\phi$ itself is both true and false. Likewise, if all variables are neither true nor false, then so is $\phi$. Nevertheless, the validity can be defined for sequents of the form $\phi\vdash\chi$ where $\phi$ and $\chi$ are formulas in the $\{\neg,\wedge,\vee\}$ language: $\phi\vdash\chi$ is valid if whenever $\phi$ is true, then $\chi$ is true as well.

$\BD$ has well-studied modal expansions with standard $\Box$- and $\lozenge$-like modalities (cf.~e.g.~\cite{Priest2008FromIftoIs,Priest2008,OdintsovWansing2017,Drobyshevich2020}). They usually employ frame semantics and use either Hilbert-style or tableaux calculi for their proof theory. There is also work on the correspondence theory for expansions of $\BD$ with $\Box$ modality and (or) some implication (cf., e.g.~\cite{RivieccioJungJansana2017,Drobyshevich2020}). In this approach, $\Box\phi$ is defined in the expected manner: $\Box\phi$ is true at $w$ iff $\phi$ is true in all accessible states; $\Box\phi$ is false at $w$ iff there is an accessible state where $\phi$ is false.

Non-standard modal expansions of $\BD$ have also been recently studied. In particular, the “classicality” operator was introduced in~\cite{AntunesCarnielliKapsnerRodriguez2020} and a~non-contingency operator $\blacktriangle\phi$ interpreted “the agent knows whether $\phi$ is true” in the epistemic contexts was proposed in~\cite{KozhemiachenkoVashentseva2023}. On the other hand, there are no (as far as the authors are aware) studies of expansions of $\BD$ with accidence or ignorance operators.
\paragraph{Plan of the paper}
The remainder of the text is organised as follows. In Section~\ref{sec:preliminaries}, we argue that the $\Box$ modality does not align well with the intuitive understanding of belief and knowledge in the framework of $\BD$. We also introduce a~non-standard modality $\blacksquare$ and provide motivation for its use as a~more suitable belief or knowledge operator). Using the intuition behind $\blacksquare$, we then show how to define the ignorance modality $\mathbf{I}$.

Section~\ref{sec:logics} is dedicated to the formal presentation of $\ignoBD$ and $\knowBD$ --- the expansions of $\BD$ with $\mathbf{I}$ and $\blacksquare$, respectively. We provide their Kripke semantics and show that $\blacktriangledown$, $\blacktriangle$ (as proposed in~\cite{KozhemiachenkoVashentseva2023}), and $\bullet$ can be defined from $\blacksquare$ in the expected manner. We also show that the $\ignoBD$-counterparts of the rules and axioms of classical logic of ignorance in~\cite{GilbertKubyshkinaPetroloVenturi2022} are valid. Moreover, we prove that $\blacksquare$ can be used as a~knowledge modality since truthfulness, positive, and negative introspection are valid on $\mathbf{S5}$-frames (and since $\mathbf{S5}$-frames are definable using $\blacktriangle$~\cite[Theorem~5.4]{KozhemiachenkoVashentseva2023}). We provide a~sound and complete tableaux calculus for $\ignoBD$ and $\knowBD$ in Section~\ref{sec:tableaux}.

Section~\ref{sec:expressivity} explores the expressivity of $\blacksquare$, $\mathbf{I}$, and $\Box$. Namely, we show that all three modalities are not mutually interdefinable. Section~\ref{sec:definability} addresses the definability of several important classes of frames in the language with $\blacksquare$. In particular, we show that Euclidean, serial, as well as transitive Euclidean frames are definable and thus $\blacksquare$ can be used as a~doxastic modality as well.

Finally, in Section~\ref{sec:conclusion}, we summarise the results of the paper and provide a~roadmap for future research.
\section{Belief, knowledge, and ignorance in Belnap--Dunn logic\label{sec:preliminaries}}
In this section, we argue that the usual definition of $\Box$ in modal expansions of $\BD$ is not well-suited for the analysis of doxastic and epistemic contexts. To do this, let us first recall\footnote{In this paper, we use Odintsov's and Wansing's~\cite{OdintsovWansing2010,OdintsovWansing2017} presentation of semantics of non-classical modal logics which uses \emph{two} valuations on a~frame --- $v^+$ (support of truth) and $v^-$ (support of falsity).} the semantics of $\KBD$ (alias $\KFDE$) from~\cite[\S11a.4]{Priest2008FromIftoIs}.
\begin{convention}[Languages]\label{conv:languages}
We fix a~countable set $\Var=\{p,q,r,\ldots\}$ of propositional variables and define the language $\Lall$ using the following grammar in Backus--Naur form.
\begin{align*}
\Lall\ni\phi&\coloneqq p\in\Var\mid\neg\phi\mid(\phi\wedge\phi)\mid(\phi\vee\phi)\mid\Box\phi\mid\blacksquare\phi\mid\mathbf{I}\phi
\end{align*}
In this paper, we will be mostly concerned with three fragments of $\Lall$ --- $\LBox$, $\Lknow$, and $\Ligno$ that contain only one modality: $\Box$, $\blacksquare$, and $\mathbf{I}$, respectively.
\end{convention}
\begin{definition}[Semantics of $\KBD$]\label{def:KBD}
A \emph{frame} is a~tuple $\mathfrak{F}=\langle W,R\rangle$ with $W\neq\varnothing$, $R$ being a~binary accessibility relation on $W$. A~\emph{model} is a~tuple $\mathfrak{M}=\langle W,R,v^+,v^-\rangle$ with $\langle W,R\rangle$ being a~frame and $v^+$ and $v^-$~being maps from $\Var$ to $2^W$ interpreted as support of truth and support of falsity, respectively. If $w\in\mathfrak{M}$, a~tuple $\langle\mathfrak{M},w\rangle$ is called a~\emph{pointed model}. We also set $R(w)=\{w':wRw'\}$ and $R^!(w)=R(w)\setminus\{w\}$.

The semantics of $\LBox$ formulas is defined as follows.
\begin{align*}
\mathfrak{M},w\Vdash^+p&\text{ iff }w\in v^+(p)&\mathfrak{M},w\Vdash^-p&\text{ iff }w\in v^-(p)\\
\mathfrak{M},w\Vdash^+\neg\phi&\text{ iff }\mathfrak{M},w\Vdash^-\phi&\mathfrak{M},w\Vdash^-\neg\phi&\text{ iff }\mathfrak{M},w\Vdash^+\phi\\
\mathfrak{M},w\Vdash^+\phi_1\!\wedge\!\phi_2&\text{ iff }\mathfrak{M},w\Vdash^+\phi_1\text{ and }\mathfrak{M},w\Vdash^+\phi_2&\mathfrak{M},w\Vdash^-\phi_1\!\wedge\!\phi_2&\text{ iff }\mathfrak{M},w\Vdash^-\phi_1\text{ or }\mathfrak{M},w\Vdash^-\phi_2\\
\mathfrak{M},w\Vdash^+\phi_1\!\vee\!\phi_2&\text{ iff }\mathfrak{M},w\Vdash^+\phi_1\text{ or }\mathfrak{M},w\Vdash^+\phi_2&\mathfrak{M},w\Vdash^-\phi_1\!\vee\!\phi_2&\text{ iff }\mathfrak{M},w\Vdash^-\phi_1\text{ and }\mathfrak{M},w\Vdash^-\phi_2\\
\mathfrak{M},w\Vdash^+\Box\phi&\text{ iff }\forall w'\in R(w):\mathfrak{M},w'\Vdash^+\phi&\mathfrak{M},w\Vdash^-\Box\phi&\text{ iff }\exists w'\in R(w):\mathfrak{M},w'\Vdash^-\phi
\end{align*}
We also write $\lozenge\phi$ as a~shorthand for $\neg\Box\neg\phi$.

Let $\mathfrak{F}$ be a~frame. $\phi\vdash\chi$ is \emph{valid on $\mathfrak{F}$} (denoted $\mathfrak{F}\models[\phi\vdash\chi]$) iff for any model $\mathfrak{M}$ on $\mathfrak{F}$, and for any $w\in\mathfrak{M}$, if $\mathfrak{M},w\Vdash^+\phi$, then $\mathfrak{M},w\Vdash^+\chi$. $\phi\vdash\chi$ is \emph{(universally) valid} iff it is valid on every frame.
\end{definition}
\begin{convention}[Notation in the models]
Throughout the paper, we are going to give examples of models. We will use the shorthands shown in Table~\ref{table:modelnotation} to denote the values of variables in states.
\begin{table}
\centering
\begin{tabular}{c|c}
\textbf{notation}&\textbf{meaning}\\\hline
$w:p^+$&$p$ is true and non-false at $w$\\
$w:p^-$&$p$ is false and non-true at $w$\\
$w:p^\pm$&$p$ is both true and false at $w$\\
$w:\xcancel{p}$&$p$ is neither true nor false at $w$
\end{tabular}
\caption{Notation in the models.}
\label{table:modelnotation}
\end{table}
\end{convention}
\begin{remark}\label{rem:4valuedBD}
In~\cite{Belnap1977fourvalued,Belnap1977computer}, $\BD$ is formulated as a~four-valued logic with truth table semantics where each value from $\{\true,\false,\both,\neither\}$\footnote{We will call these values “Belnapian”.} represents what a~computer or a~database might be told regarding a~given statement.
\begin{itemize}
\item $\true$ stands for “just told True”.
\item $\false$ stands for “just told False”.
\item $\both$ (or \textbf{Both}) stands for “told both True and False”.
\item $\neither$ (or \textbf{None}) stands for “told neither True nor False”.
\end{itemize}
It is also possible to formulate $\KBD$ as a~\emph{four-valued} modal logic (cf., e.g.,~\cite{Priest2008,Priest2008FromIftoIs}).
\end{remark}

Now, why is $\Box$ not well-suited to formalise belief or knowledge? Recall from classical logic that $\triangle\phi$ which is read as “$\phi$ is non-contingent”, “the agent knows whether $\phi$ is the case” (in the epistemic setting), or “the agent is opinionated w.r.t.\ $\phi$” (in the doxastic setting) is true at a~given state $w$ when $\phi$ has \emph{the same value in all accessible states}. Thus, $\Box\phi\rightarrow\triangle\phi$\footnote{Read “if the agent knows that $\phi$ is true, they know whether $\phi$ is true” or “if the agent believes in $\phi$, they are opinionated w.r.t.\ $\phi$”.} is a~valid formula. Classically, this is evident since $\triangle\phi\coloneqq\Box\phi\vee\Box\neg\phi$. Moreover, if the underlying frame is reflexive, $\triangle$~can be used to define~$\Box$: $\Box\phi\coloneqq\phi\wedge\triangle\phi$. This, however, is not necessarily the case in~$\BD$.

Indeed, $\BD$ can be viewed as a~four-valued logic (Remark~\ref{rem:4valuedBD}). Just as in classical logic, it is reasonable to consider “$\phi$ is non-contingent” true only if $\phi$ has the same value in all accessible states. This time, however, $\phi$ can have not two but four values. Note, however, that it is possible for $\Box p$ to be true at a~given state even when $p$ has different values in different states accessible from~$w$ (cf.~Fig.~\ref{fig:trueandfalse}). Thus, $\Box\phi\vee\Box\neg\phi$ is not a~suitable formalisation of the knowing whether modality in~$\BD$.
\begin{figure}
\centering
\begin{tikzpicture}[modal,node distance=1cm,world/.append style={minimum
size=1cm}]
\node[world] (w0) [label=below:{$w_0$}] {$p^+$};
\node[world] (w1) [right=of w0] [label=below:{$w_1$}] {$p^\pm$};
\node[] [left=of w0] {$\mathfrak{M}$:};
\path[->] (w0) edge[reflexive] (w0);
\path[->] (w0) edge (w1);
\end{tikzpicture}
\caption{$\mathfrak{M},w_0\Vdash^+\Box p$.}
\label{fig:trueandfalse}
\end{figure}
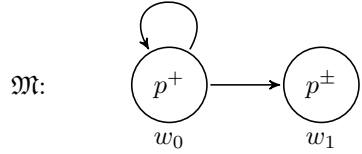

In~\cite{KozhemiachenkoVashentseva2023}, we were addressing this issue and introduced a~non-contingency modality $\blacktriangle$ that captures the intuition behind “knowing whether” in a~four-valued setting better than $\Box\phi\vee\Box\neg\phi$. Namely, for $\blacktriangle\phi$ to be true at $w$, $\phi$ should have the same Belnapian value in every accessible state (thus, $\mathfrak{M},w_0\nVdash^+\blacktriangle p$ in Fig.~\ref{fig:trueandfalse}). Now, we can introduce a~new doxastic or epistemic modality $\blacksquare$ that satisfies the above desideratum. Namely, in order for $\blacksquare\phi$ to be true at $w$, not only should $\phi$ be true at all accessible states but $\phi$ should have the same (Belnapian) value in all of them. In particular, it should be the case that $\mathfrak{M},w_0\nVdash^+\blacksquare p$.

Let us now present the intuitions behind the ignorance modality $\mathbf{I}$. First, we recall the \emph{classical} ignorance modality $\mathbb{I}$ as defined in~\cite{KubyshkinaPetrolo2021}. A \emph{classical} Kripke model\footnote{We refer our readers to~\cite{BlackburndeRijkeVenema2010} for the detailed presentation of the classical semantics of modal logic.} is a~tuple $\mathfrak{M}=\langle W,R,v\rangle$ with $W\neq\varnothing$, $R\subseteq W\times W$ and $v$ being a classical valuation. The semantics of $\mathbb{I}$ is then as follows:
\begin{align}
\mathfrak{M},w\Vdash\mathbb{I}\phi&\text{ iff }\mathfrak{M},w\Vdash\phi\text{ and }\forall w'\in R^!(w):\mathfrak{M},w'\nVdash\phi\label{equ:ignoclassical}
\end{align}
I.e., the agent is ignorant of $\phi$ when $\phi$ is true but the agent believes\footnote{This is very close to the “being wrong” modality $\mathbb{W}$ proposed in~\cite{Steinsvold2011} that is defined as follows: \[\mathfrak{M},w\Vdash\mathbb{W}\phi\text{ iff }\mathfrak{M},w\nVdash\phi\text{ and }\forall w'\in R(w):\mathfrak{M},w'\Vdash\phi\] I.e., $\phi$ is false but the agent believes that it is true. Note, however, that in contrast to $\mathbb{I}$, $\mathbb{W}$ uses the whole accessibility relation and \emph{does not exclude} $w$, i.e., $\mathbb{W}\phi\coloneqq\neg\phi\wedge\Box\phi$. Thus, $\mathbb{W}\phi$ is \emph{always false} on a~reflexive frame. We direct the reader to~\cite{GilbertKubyshkinaPetroloVenturi2022} for a~detailed comparison between $\mathbb{I}$ and $\mathbb{W}$.} that it is false (if they take into account accessible states that are different from $w$). Note that while $\mathbb{I}$ is not definable via $\Box$~\cite[\S3]{KubyshkinaPetrolo2021}, it is convenient to represent $\mathbb{I}\phi$ as $\phi\wedge\Box^!\neg\phi$ where $\Box^!$ is a~doxastic or epistemic modality w.r.t.~$R^!(w)$ (and not $R(w)$):
\begin{align}
\mathfrak{M},w\Vdash\Box^!\phi&\text{ iff }\forall w'\in R^!(w):\mathfrak{M},w'\Vdash\phi\label{equ:strictboxclassical}
\end{align}
\section{Belnap--Dunn logics of knowledge and ignorance\label{sec:logics}}
Let us now formalise the accounts of $\blacksquare$ and $\mathbf{I}$ given in the previous section. We begin with the presentation of $\knowBD$ and $\ignoBD$ and then discuss their semantical properties.
\subsection{Language and semantics\label{ssec:ignoknowlanguage}}
The next definition presents the Kripke semantics of $\blacksquare$ and $\mathbf{I}$.
\begin{definition}[Kripke semantics of $\knowBD$ and $\ignoBD$]\label{def:ignoknowsemantics}
Let $\mathfrak{M}=\langle W,R,v^+,v^-\rangle$ be a model as presented in Definition~\ref{def:KBD}. We define the semantics of $\Lknow$ and $\Ligno$ formulas as follows: the truth and falsity conditions of propositional formulas are as in Definition~\ref{def:KBD}; the semantics of $\blacksquare\phi$ and $\mathbf{I}\phi$ are given below.
\begin{align*}
\mathfrak{M},w\Vdash^+\blacksquare\phi\text{ iff }&\forall w'\in R(w):\mathfrak{M},w'\Vdash^+\phi\text{ and }\forall w_1,w_2\in R(w):\mathfrak{M},w_1\Vdash^-\phi\Rightarrow\mathfrak{M},w_2\Vdash^-\phi\\
\mathfrak{M},w\Vdash^-\blacksquare\phi\text{ iff }&\exists w'\in R(w):\mathfrak{M},w'\Vdash^-\phi\text{ or }\exists w_1,w_2\in R(w):\mathfrak{M},w_1\Vdash^+\phi\text{ and }\mathfrak{M},w_2\nVdash^+\phi\\
\mathfrak{M},w\Vdash^+\mathbf{I}\phi\text{ iff }&\mathfrak{M},w\Vdash^+\phi\text{ and }\forall w'\in R^!(w):\mathfrak{M},w'\Vdash^-\phi\\&\text{ and }\forall w_1,w_2\in R^!(w):\mathfrak{M},w_1\Vdash^+\phi\Rightarrow\mathfrak{M},w_2\Vdash^+\phi\\
\mathfrak{M},w\Vdash^-\mathbf{I}\phi\text{ iff }&\mathfrak{M},w\Vdash^-\phi\text{ or }\exists w'\in R^!(w):\mathfrak{M},w'\Vdash^+\phi\\&\text{ or }\exists w'_1,w'_2\in R^!(w):\mathfrak{M},w'_1\nVdash^-\phi\text{ and }\mathfrak{M},w'_2\Vdash^-\phi
\end{align*}
We also write $\blacklozenge\phi$ as a~shorthand for $\neg\blacksquare\neg\phi$ and $\bullet\phi$ as a~shorthand for $\phi\wedge\neg\blacksquare\phi$.

The validity is defined as expected. For $\phi,\chi\in\Lknow$ ($\phi,\chi\in\Ligno$, respectively), $\phi\vdash\chi$ is \emph{valid on a~frame $\mathfrak{F}$} (denoted $\mathfrak{F}\models[\phi\vdash\chi]$) iff for any model $\mathfrak{M}$ on $\mathfrak{F}$, and for any $w\in\mathfrak{M}$, if $\mathfrak{M},w\Vdash^+\phi$, then $\mathfrak{M},w\Vdash^+\chi$. $\phi\vdash\chi$ is \emph{$\knowBD$ valid} (\emph{$\ignoBD$ valid}, respectively) iff it is valid on every frame.
\end{definition}
\begin{remark}\label{rem:ignobox!}
One can notice that, indeed, given a~frame $\mathfrak{F}=\langle W,R\rangle$ and a~pointed model $\langle\mathfrak{M},w\rangle$ on it, we have
\begin{align}\label{equ:ignobox!}
\mathfrak{M},w\Vdash^+\mathbf{I}p&\text{ iff }\mathfrak{M},w\Vdash^+p\wedge\blacksquare^!\neg p&\mathfrak{M},w\Vdash^-\mathbf{I}p&\text{ iff }\mathfrak{M},w\Vdash^-p\wedge\blacksquare^!\neg p
\end{align}
where $\blacksquare^!$ is associated to $R^!$ (recall~\eqref{equ:ignoclassical} and~\eqref{equ:strictboxclassical} as well as Definition~\ref{def:KBD}).

Moreover, it is easy to see from Definitions~\ref{def:KBD} and~\ref{def:ignoknowsemantics} that as long as all formulas have classical values in all states\footnote{I.e., there is no formula $\phi$ and no state $w$ s.t.\ one of the following holds:\begin{itemize}\item$\mathfrak{M},w\Vdash^+\phi$ and $\mathfrak{M},w\Vdash^-\phi$, or\item$\mathfrak{M},w\nVdash^+\phi$ and $\mathfrak{M},w\nVdash^-\phi$.\end{itemize}} of a~given model, then $\mathbf{I}$, $\blacksquare$, and $\bullet$ \emph{behave classically}.
\end{remark}
It is also clear that there are no valid formulas in $\KBD$, $\knowBD$, and $\ignoBD$ as the following proposition states.
\begin{proposition}\label{prop:novalidities}
Let $\phi\!\in\!\Lall$. Then, there are pointed models $\langle\mathfrak{M},w\rangle$ and $\langle\mathfrak{M}',w'\rangle$ s.t.\ $\mathfrak{M},w\nVdash^+\phi$ and $\mathfrak{M}',w'\Vdash^-\phi$.
\end{proposition}
\begin{proof}
We construct two pointed models: the one where \emph{every} $\Lall$-formula is \emph{false}, and the other where every $\Lall$-formula is \emph{non-true}. Consider Fig.~\ref{fig:BandNall}. It is easy to check by induction that for any $\phi\in\Lall$ (1) $\mathfrak{M},w_0\Vdash^+\phi$ and $\mathfrak{M},w_0\Vdash^-\phi$; (2) $\mathfrak{M},w'_0\nVdash^+\phi$ and $\mathfrak{M},w'_0\nVdash^-\phi$. The result follows.
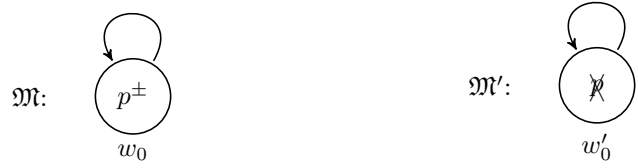
\begin{figure}
\centering
\begin{tikzpicture}[modal,node distance=0.5cm,world/.append style={minimum
size=1cm}]
\node[world] (w0) [label=below:{$w_0$}] {$p^\pm$};
\node[] [left=of w0] {$\mathfrak{M}$:};
\path[->] (w0) edge[reflexive] (w0);
\end{tikzpicture}
\hfil
\begin{tikzpicture}[modal,node distance=0.5cm,world/.append style={minimum
size=1cm}]
\node[world] (w0) [label=below:{$w'_0$}] {$\xcancel{p}$};
\node[] [left=of w0] {$\mathfrak{M}'$:};
\path[->] (w0) edge[reflexive] (w0);
\end{tikzpicture}
\caption{All variables have the same value in both models as exemplified by $p$.}
\label{fig:BandNall}
\end{figure}
\end{proof}

Note that all logics we are considering here --- $\KBD$, $\knowBD$, and $\ignoBD$ --- are conservative expansions of $\BD$. This, however, is not sufficient to obtain Proposition~\ref{prop:novalidities} since it is possible to define modalities in such a way that \emph{modal formulas} can be valid.

Let us now show a~technical result analogous to~\cite[Lemma~2.12]{KozhemiachenkoVashentseva2023} that will simplify some proofs in this section.
\begin{definition}[Dual models]\label{dualvaluationsdefinition}
For any model $\mathfrak{M}=\langle W,R,v^+,v^-\rangle$, we define its \emph{dual model} on the same frame $\mathfrak{M}_\partial=\langle W,R,v^+_\partial,v^-_\partial\rangle$ as follows.
\begin{align*}
\text{if }w\in v^+(p)\text{ and }w\notin v^-(p)&\text{ then }w\in v^+_\partial(p)\text{ and }w\notin v^-_\partial(p)\\
\text{if }w\in v^+(p)\text{ and }w\in v^-(p)&\text{ then }w\notin v^+_\partial(p)\text{ and }w\notin v^-_\partial(p)\\
\text{if }w\notin v^+(p)\text{ and }w\notin v^-(p)&\text{ then }w\in v^+_\partial(p)\text{ and }w\in v^-_\partial(p)\\
\text{if }w\notin v^+(p)\text{ and }w\in v^-(p)&\text{ then }w\notin v^+_\partial(p)\text{ and }w\in v^-_\partial(p)
\end{align*}
\end{definition}
In other words, if a~variable was either true and non-false or false and non-true in some state in a~model, then it remains such in the dual\footnote{Note that the dual model swaps $\both$ and $\neither$ which mimics the behaviour of \emph{conflation} presented in~\cite{Fitting1994}. We chose against ‘conflated model’ for two reasons: first, to preserve the terminology from~\cite{KozhemiachenkoVashentseva2023} where such models are also called ‘dual’. Second, ‘conflated model’ might sound confusing. Note, finally, that $(\mathfrak{M}^\partial)^\partial=\mathfrak{M}$.} model. But if it was both true and false, it becomes neither true nor false and vice versa.
\begin{proposition}\label{prop:dualvaluations}
Let $\mathfrak{M}=\langle W,R,v^+,v^-\rangle$ be a~model and $\mathfrak{M}_\partial=\langle W,R,v^+_\partial,v^-_\partial\rangle$ be its dual model. Then for any $\phi\in\Ligno\cup\Lknow$ and $w\in\mathfrak{M}$, it holds that
\begin{align*}
\text{if }\mathfrak{M},w\Vdash^+\phi\text{ and }\mathfrak{M},w\nVdash^-\phi&\text{ then }\mathfrak{M}_\partial,w\Vdash^+\phi\text{ and }\mathfrak{M}_\partial,w\nVdash^-\phi\\
\text{if }\mathfrak{M},w\Vdash^+\phi\text{ and }\mathfrak{M},w\Vdash^-\phi&\text{ then }\mathfrak{M}_\partial,w\nVdash^+\phi\text{ and }\mathfrak{M}_\partial,w\nVdash^-\phi\\
\text{if }\mathfrak{M},w\nVdash^+\phi\text{ and }\mathfrak{M},w\nVdash^-\phi&\text{ then }\mathfrak{M}_\partial,w\Vdash^+\phi\text{ and }\mathfrak{M}_\partial,w\Vdash^-\phi\\
\text{if }\mathfrak{M},w\nVdash^+\phi\text{ and }\mathfrak{M},w\Vdash^-\phi&\text{ then }\mathfrak{M}_\partial,w\nVdash^+\phi\text{ and }\mathfrak{M}_\partial,w\Vdash^-\phi
\end{align*}
\end{proposition}
\begin{proof}
We adapt the technique from~\cite{ZaitsevShramko2004english} and prove the statement by induction on $\phi$. The basis case of propositional variables holds by the construction of $v^+_\partial$ and $v^-_\partial$. The cases of propositional connectives hold by virtue of the admissibility of the contraposition in $\BD$~\cite{Font1997,Dunn2000,ZaitsevShramko2004english}. It remains to consider the cases of $\blacksquare$ and $\mathbf{I}$.

Let $\mathfrak{M},w\Vdash^+\blacksquare\phi$ and $\mathfrak{M},w\nVdash^-\blacksquare\phi$. Then, $\mathfrak{M},w'\Vdash^+\phi$ and $\mathfrak{M},w'\nVdash^-\phi$ in all $w'\in R(w)$. Applying the induction hypothesis, we have that $\mathfrak{M}_\partial,w'\Vdash^+\phi$ and $\mathfrak{M}_\partial,w'\nVdash^-\phi$ in all $w'\in R(w)$, whence $\mathfrak{M}_\partial,w\Vdash^+\blacksquare\phi$ and $\mathfrak{M}_\partial,w\nVdash^-\blacksquare\phi$.

Now assume that $\mathfrak{M},w\nVdash^+\blacksquare\phi$ and $\mathfrak{M},w\Vdash^-\blacksquare\phi$. Then, there are the following cases.
\begin{enumerate}
\item[$(a)$] There is $w'\in R(w)$ s.t.\ $\mathfrak{M},w'\nVdash^+\phi$ and $\mathfrak{M},w'\Vdash^-\phi$.
\item[$(b)$] There are $w_1,w_2\in R(w)$ s.t.\ one of the following holds:
\begin{enumerate}
\item[$(b.1)$] $\mathfrak{M},w_1\Vdash^+\phi$ and $\mathfrak{M},w_1\nVdash^-\phi$ but $\mathfrak{M},w_2\Vdash^+\phi$ and $\mathfrak{M},w_2\Vdash^-\phi$;
\item[$(b.2)$] $\mathfrak{M},w_1\Vdash^+\phi$ and $\mathfrak{M},w_1\nVdash^-\phi$ but $\mathfrak{M},w_2\nVdash^+\phi$ and $\mathfrak{M},w_2\nVdash^-\phi$;
\item[$(b.3)$] $\mathfrak{M},w_1\Vdash^+\phi$ and $\mathfrak{M},w_1\Vdash^-\phi$ but $\mathfrak{M},w_2\nVdash^+\phi$ and $\mathfrak{M},w_2\nVdash^-\phi$.
\end{enumerate}
\end{enumerate}
Applying the induction hypothesis, we obtain the following.
\begin{enumerate}
\item[$(a')$] There is $w'\in R(w)$ s.t.\ $\mathfrak{M},w'\nVdash^+\phi$ and $\mathfrak{M},w'\Vdash^-\phi$ (nothing changes from $(a)$).
\item[$(b')$] There are $w_1,w_2\in R(w)$ s.t.\ one of the following holds:
\begin{enumerate}
\item[$(b'.1)$] $\mathfrak{M},w_1\Vdash^+\phi$ and $\mathfrak{M},w_2\nVdash^-\phi$ but $\mathfrak{M},w_1\nVdash^+\phi$ and $\mathfrak{M},w_2\nVdash^-\phi$;
\item[$(b'.2)$] $\mathfrak{M},w_1\Vdash^+\phi$ and $\mathfrak{M},w_2\nVdash^-\phi$ but $\mathfrak{M},w_1\Vdash^+\phi$ and $\mathfrak{M},w_2\Vdash^-\phi$;
\item[$(b'.3)$] $\mathfrak{M},w_1\nVdash^+\phi$ and $\mathfrak{M},w_1\nVdash^-\phi$ but $\mathfrak{M},w_2\nVdash^+\phi$ and $\mathfrak{M},w_2\nVdash^-\phi$.
\end{enumerate}
\end{enumerate}
It is clear that in all cases: $(a')$ and $(b'.1)$--$(b'.3)$, it holds that $\mathfrak{M}_\partial,w\nVdash^+\blacksquare\phi$ and $\mathfrak{M}_\partial,w\Vdash^-\blacksquare\phi$.

The cases where $\mathfrak{M},w\Vdash^+\blacksquare\phi$ and $\mathfrak{M},w\Vdash^-\blacksquare\phi$ or $\mathfrak{M},w\nVdash^+\blacksquare\phi$ and $\mathfrak{M},w\nVdash^-\blacksquare\phi$ can be tackled similarly.

Let us proceed to $\mathbf{I}\phi$. Observe from Remark~\ref{rem:ignobox!} and~\eqref{equ:ignobox!} that $\mathbf{I}\phi$ can be defined as $\phi\wedge\blacksquare^!\phi$. Since the statement holds for $\blacksquare$ and propositional connectives on every frame $\mathfrak{F}=\langle W,R\rangle$ and since $\blacksquare^!$ is just $\blacksquare$ defined with $R^!(w)$ instead of $R(w)$, we obtain the result for $\mathbf{I}\phi$ as well.
\end{proof}
\begin{remark}\label{rem:equivalence}
Proposition~\ref{prop:dualvaluations} has an important immediate consequence: if both $\phi\vdash\chi$ and $\chi\vdash\phi$ are valid (on a~given frame), then
\begin{align*}
\mathfrak{M},w\Vdash^+\phi&\text{ iff }\mathfrak{M},w\Vdash^+\chi&\mathfrak{M},w\Vdash^-\phi&\text{ iff }\mathfrak{M},w\Vdash^-\chi
\end{align*}
for every pointed model $\langle\mathfrak{M},w\rangle$ (on that frame). Moreover, it follows that if $\phi\vdash\chi$ is valid (on a~given frame), then $\neg\chi\vdash\neg\phi$ is also valid (on that frame). I.e., the contraposition is sound, as expected.
\end{remark}
\subsection{Semantical properties of $\blacksquare$\label{ssec:knowsemantics}}
Let us now show that $\blacksquare$ conforms to the intuitions outlined in Section~\ref{sec:preliminaries}. We begin with recalling the semantics of $\blacktriangle$ from~\cite{KozhemiachenkoVashentseva2023}.
\begin{definition}[Semantics of $\blacktriangle$]\label{def:contingencysemantics}
Let $\mathfrak{M}=\langle W,R,v^+,v^-\rangle$ be a model as presented in Definition~\ref{def:KBD}. To make the presentation of the semantics for $\blacktriangle$ more concise we introduce the following conditions.
\begin{align*}
\forall w_1,w_2\in R(w_0):(\mathfrak{M},w_1\!\Vdash^+\!\phi\!\Rightarrow\!\mathfrak{M},w_2\!\Vdash^+\!\phi)~\&~(\mathfrak{M},w_1\!\Vdash^-\!\phi\!\Rightarrow\!\mathfrak{M},w_2\Vdash^-\phi)\tag{$t_1\blacktriangle$}\label{t1conditionI}\\
\forall w_1\in R(w_0):\mathfrak{M},w_1\Vdash^+\phi\text{ or }\mathfrak{M},w_1\Vdash^-\phi\tag{$t_2\blacktriangle$}\label{t2conditionI}\\
\exists w_1,w_2\in R(w_0):\mathfrak{M},w_1\Vdash^+\phi~\&~\mathfrak{M},w_2\nVdash^+\phi\tag{$f_1\blacktriangle$}\label{f1conditionI}\\
\exists w_1,w_2\in R(w_0):\mathfrak{M},w_1\Vdash^-\phi~\&~\mathfrak{M},w_2\nVdash^-\phi\tag{$f_2\blacktriangle$}\label{f2conditionI}\\
\exists w_1,w_2\in R(w_0):\mathfrak{M},w_1\Vdash^+\phi~\&~\mathfrak{M},w_2\Vdash^-\phi\tag{$f_3\blacktriangle$}\label{fconditionS}
\end{align*}

Using these conditions, support of truth and support of falsity of $\blacktriangle$ is defined as follows.
\begin{align*}
\mathfrak{M},w_0\Vdash^+\blacktriangle\phi&\text{ iff }\eqref{t1conditionI}\text{ and }\eqref{t2conditionI}&\mathfrak{M},w_0\Vdash^-\blacktriangle\phi&\text{ iff }\eqref{f1conditionI}\text{ or }\eqref{f2conditionI}\text{ or }\eqref{fconditionS}
\end{align*}
We also write $\blacktriangledown\phi$ as a~shorthand for $\neg\blacktriangle\phi$.
\end{definition}

The next statement shows that $\blacksquare$ behaves in the desired way. Namely, it has the expected connection with the “knowledge whether” modality and, in addition, truthfulness, positive introspection, and negative introspection are valid on $\mathbf{S5}$ frames (i.e., frames $\langle W,R\rangle$ where $R$ is an equivalence relation).
\begin{theorem}\label{theorem:knowisgood}~
\begin{enumerate}
\item $\blacktriangle p\dashv\vdash\blacksquare p\vee\blacksquare\neg p$ is valid on every frame.
\item $\blacksquare p\dashv\vdash p\wedge\blacktriangle p$ is valid on every \emph{reflexive} frame.
\item Let $\mathfrak{F}=\langle W,R\rangle$ be an \emph{$\mathbf{S5}$ frame}. Then $\blacksquare p\vdash p$, $\blacksquare p\vdash\blacksquare\blacksquare p$, and $\blacklozenge p\vdash\blacksquare\blacklozenge p$ are valid on $\mathfrak{F}$.
\end{enumerate}
\end{theorem}
\begin{proof}
The proofs of 1.\ and 2.\ are immediate from Definitions~\ref{def:ignoknowsemantics} and~\ref{def:contingencysemantics}. Let us prove 3. Let $\mathfrak{F}=\langle W,R\rangle$ be an $\mathbf{S5}$ frame and $w\in W$. Now let $\mathfrak{M}$ be a~model on $\mathfrak{F}$ s.t.\ $\mathfrak{M},w\Vdash^+\blacksquare p$. Thus, $\mathfrak{M},w'\Vdash^+p$ in every $w'\in R(w)$. But since $R$ is reflexive, we have that $\mathfrak{M},w\Vdash^+p$ and thus, $\blacksquare p\vdash p$ is valid.

To prove the validity of $\blacksquare p\vdash\blacksquare\blacksquare p$, let again $\mathfrak{M},w\Vdash^+\blacksquare p$ for some model $\mathfrak{M}$ on $\mathfrak{F}$. We consider two cases. First, if $\mathfrak{M},w\nVdash^-\blacksquare p$, then $\mathfrak{M},w'\Vdash^+p$ and $\mathfrak{M},w'\nVdash^-p$ in every $w'\in R(w)$. Thus, since $R$ is transitive, we have that $\mathfrak{M},w'\Vdash^+\blacksquare p$ and $\mathfrak{M},w'\nVdash^-\blacksquare p$ in every $w'\in R(w)$. Second, let $\mathfrak{M},w\nVdash^+\blacksquare p$. Then, $\mathfrak{M},w'\Vdash^+p$ and $\mathfrak{M},w'\Vdash^-p$ in every $w'\in R(w)$. Thus, since $R$ is reflexive\footnote{Reflexivity is crucial here since $R(u)\neq\varnothing$ for every $u\in W$ which guarantees that $\blacksquare p$ is \emph{both true and false} in every state.} and transitive, we have that $\mathfrak{M},w'\Vdash^+\blacksquare p$ and $\mathfrak{M},w'\Vdash^-\blacksquare p$ in every $w'\in R(w)$. In the first case, we obtain that $\mathfrak{M},w\Vdash^+\blacksquare\blacksquare p$ and $\mathfrak{M},w\nVdash^-\blacksquare\blacksquare p$. In the second case, we have that $\mathfrak{M},w\Vdash^+\blacksquare\blacksquare p$ and $\mathfrak{M},w\Vdash^-\blacksquare\blacksquare p$. We can now conclude that $\blacksquare p\vdash\blacksquare\blacksquare p$ is valid.

Finally, let $\mathfrak{M},w\Vdash^+\blacklozenge p$. We have four cases:
\begin{enumerate}
\item there is $w'\in R(w)$ s.t.\ $\mathfrak{M},w'\Vdash^+p$ and $\mathfrak{M},w'\nVdash^-p$;
\item there are $w_1,w_2\in R(w)$ s.t.\ $\mathfrak{M},w_1\Vdash^-p$ and $\mathfrak{M},w_2\nVdash^-p$;
\item there are $w_1,w_2\in R(w)$ s.t.\ $\mathfrak{M},w_1\Vdash^+p$ and $\mathfrak{M},w_2\nVdash^+p$;
\item $\mathfrak{M},w'\Vdash^+p$ and $\mathfrak{M},w'\Vdash^-p$ in every $w'\in R(w)$.
\end{enumerate}
One can see that $\mathfrak{M},w\nVdash^-\blacklozenge p$ in cases 1--3 and $\mathfrak{M},w\Vdash^-\blacklozenge p$ in case 4. Now, since $R$ is Euclidean, it is clear that $uRu'$ for every $u,u'\in R(w)$. Thus, in cases 1--3, we have that $\mathfrak{M},t\Vdash^+\blacklozenge p$ and $\mathfrak{M},t\nVdash^-\blacklozenge p$ for every $t\in R(w)$, and in case 4., $\mathfrak{M},t\Vdash^+\blacklozenge p$ and $\mathfrak{M},t\Vdash^-\blacklozenge p$ for every $t\in R(w)$. Hence, $\mathfrak{M},w\Vdash^+\blacksquare\blacklozenge p$, as required and thus, $\blacklozenge p\vdash\blacksquare\blacklozenge p$ is valid on $\mathfrak{F}$.
\end{proof}
\begin{remark}\label{rem:knowS5isadequate}
The above statement shows that $\blacksquare$ fulfils the desiderata w.r.t.\ a~knowledge modality. First, it has the expected connections with the “knowledge whether” ($\blacktriangle$). Second, truthfulness, positive introspection, and negative introspection are valid on $\mathbf{S5}$ frames. In addition, $\mathbf{S5}$ frames are definable using $\blacktriangle$~\cite[Theorem~5.4]{KozhemiachenkoVashentseva2023} (and thus, using $\blacksquare$ as well). We will later see (cf.~Section~\ref{sec:definability}) that $\blacksquare$ can define several important doxastic classes of frames in a~natural way and thus can act as a~belief modality as well.
\end{remark}

\begin{remark}\label{rem:knownonstandard}
Note, however, that $\blacksquare$ is not a~standard modality in contrast to $\Box$. Indeed, while $\Box(p\wedge q)\vdash\Box p\wedge\Box q$ is valid on every frame, one can check that $\blacksquare(p\wedge q)\vdash\blacksquare p\wedge\blacksquare q$ is valid only on partial-functional frames (i.e., frames where $|R(w)|\leq1$ for every $w$). This is an expected consequence of its semantics since we demand that the Belnapian value of $\phi$ be the same in all accessible states for $\blacksquare\phi$ to be true or non-false.

Still, it is easy to see that
\begin{align}\label{equ:exactlytrueknow}
\mathfrak{M},w\Vdash^+\blacksquare(p\wedge q)\text{ and }\mathfrak{M},w\nVdash^-\blacksquare(p\wedge q)&\text{ iff }\mathfrak{M},w\Vdash^+\blacksquare p\wedge\blacksquare q\text{ and }\mathfrak{M},w\nVdash^-\blacksquare p\wedge\blacksquare q
\end{align}
I.e., to refute $\blacksquare(p\wedge q)\vdash\blacksquare p\wedge\blacksquare q$, one needs to assume that $\blacksquare(p\wedge q)$ is \emph{both true and false at $w$}. This is reasonable since the truth condition on $\blacksquare$ is stronger than that on $\Box$ (cf.~Definitions~\ref{def:KBD} and~\ref{def:ignoknowsemantics}). Furthermore, observe that if $\blacksquare\phi$ is \emph{both true and false at $w$}, it means that the agent believes (or knows) that $\phi$ is contradictory or paradoxical (since it has to be both true and false in all accessible states). But if $\phi=p\wedge q$, it makes sense to argue that the agent might not even have an opinion whether it is $p$, $q$, or both of them that have paradoxical value. In fact, one can maintain that in the presence of paradoxical truth-values, the belief is not compositional even w.r.t.\ conjunction~\cite{Dubois2008}.
\end{remark}
\subsection{Semantical properties of $\mathbf{I}$\label{ssec:ignosemantics}}
Let us now proceed to the Belnap--Dunn ignorance modality. To further motivate our semantics of~$\mathbf{I}$, we show that the counterparts of the modal axioms and rules presented in~\cite[Definition~1.12]{GilbertKubyshkinaPetroloVenturi2022} are valid. These are as follows:
\begin{align}
(\mathbb{I}1)~\mathbb{I}p\rightarrow p&&(\mathbb{I}2)~(\mathbb{I}p\wedge\mathbb{I}q)\rightarrow\mathbb{I}(p\vee q)&&(\mathbb{I}R)~\dfrac{\phi\rightarrow\chi}{\phi\rightarrow(\mathbb{I}\chi\rightarrow\mathbb{I}\phi)}
\label{equ:classicalignoaxioms}
\end{align}
To produce their $\ignoBD$-counterparts, we use the fact that $\phi\rightarrow(\mathbb{I}\chi\rightarrow\mathbb{I}\phi)$ is equivalent to $(\phi\wedge\mathbb{I}\chi)\rightarrow\mathbb{I}\phi$ in classical logic and then replace $\mathbb{I}$ with $\mathbf{I}$ and $\rightarrow$ with $\vdash$. This gives us the following:
\begin{align}
(\mathbf{I}1)~\mathbf{I}p\vdash p&&(\mathbf{I}2)~\mathbf{I}p\wedge\mathbf{I}q\vdash\mathbf{I}(p\vee q)&&(\mathbf{I}R)~\dfrac{\phi\vdash\chi}{\phi\wedge\mathbf{I}\chi\vdash\mathbf{I}\phi}
\label{equ:ignoaxioms}
\end{align}
\begin{theorem}
\label{theorem:ignoisgood}
All sequents ($\mathbf{I}1$ and $\mathbf{I}2$) and the rule ($\mathbf{I}R$) in~\eqref{equ:ignoaxioms} are valid on every frame.
\end{theorem}
\begin{proof}
The validity of $\mathbf{I}1$ is evident from Definition~\ref{def:ignoknowsemantics}. Let us consider $\mathbf{I}2$. Assume that $\mathfrak{F}$ is a~frame and $\langle\mathfrak{M},w\rangle$ is a~pointed model on $\mathfrak{F}$ s.t.\ $\mathfrak{M},w\Vdash^+\mathbf{I}p\wedge\mathbf{I}q$. Then, we have that $\mathfrak{M},w\Vdash^+p$ and $\mathfrak{M},w\Vdash^+q$, that $\mathfrak{M},w'\Vdash^-p$ and $\mathfrak{M},w'\Vdash^-q$ in every $w'\in R^!(w)$, and, furthermore, that if $\mathfrak{M},u\Vdash^+p$ in some $u\in R^!(w)$, then $\mathfrak{M},u'\Vdash^+p$ in \emph{every} $u'\in R^!(w)$ (and likewise, if $\mathfrak{M},t\Vdash^+q$ in some $t\in R^!(w)$, then $\mathfrak{M},t'\Vdash^+q$ in \emph{every} $t'\in R^!(w)$). Thus, we have that $\mathfrak{M},w\Vdash^+p\vee q$, $\mathfrak{M},w'\Vdash^-p\vee q$ in every $w'\in R^!(w)$, and, in addition, if $\mathfrak{M},t\Vdash^-p\vee q$ in some $t\in R^!(w)$, then $\mathfrak{M},t'\Vdash^-p\vee q$ in \emph{every} $t'\in R^!(w)$. Hence, $\mathfrak{M},w\Vdash^+\mathbf{I}(p\vee q)$, as required.

To tackle $\mathbf{I}R$, we proceed by contraposition and assume that $\phi\wedge\mathbf{I}\chi\vdash\mathbf{I}\phi$ is not valid. Namely, that there is a~pointed model $\langle\mathfrak{M},w\rangle$ s.t.\ $\mathfrak{M},w\Vdash^+\phi\wedge\mathbf{I}\chi$ but $\mathfrak{M},w\nVdash^+\mathbf{I}\phi$. We consider two cases: (1) $\mathfrak{M},w\nVdash^-\mathbf{I}\phi$, and (2) $\mathfrak{M},w\Vdash^-\mathbf{I}\phi$.

If (1), we have that $\mathfrak{M},w'\nVdash^+\phi$ and $\mathfrak{M},w'\nVdash^-\phi$ in all $w'\in R^!(w)$ (since $\mathfrak{M}_\partial,w\Vdash^+\phi$). On the other hand, $\mathfrak{M},w'\Vdash^-\chi$ in all $w'\in R^!(w)$. Applying Proposition~\ref{prop:dualvaluations}, we obtain that $\mathfrak{M}_\partial,w'\Vdash^+\phi$ but $\mathfrak{M},w'\nVdash^-\chi$ in all $w'\in R^!(w)$, i.e., $\phi\vdash\chi$ is not valid.

If (2), one of the following holds:
\begin{enumerate}
\item[$(2.a)$] there is $w'\in R^!(w)$ s.t.\ $\mathfrak{M},w'\Vdash^+\phi$ and $\mathfrak{M},w'\nVdash^-\phi$;
\item[$(2.b)$] there are $w',w''\in R^!(w)$ s.t.\ $\mathfrak{M},w'\Vdash^+\phi$ and $\mathfrak{M},w'\Vdash^-\phi$ but $\mathfrak{M},w''\nVdash^+\phi$ and $\mathfrak{M},w''\nVdash^-\phi$;
\item[$(2.c)$] there are $w',w''\in R^!(w)$ s.t.\ $\mathfrak{M},w'\Vdash^+\phi$ and $\mathfrak{M},w'\Vdash^-\phi$ but $\mathfrak{M},w''\nVdash^+\phi$ and $\mathfrak{M},w''\Vdash^-\phi$;
\item[$(2.d)$] there are $w',w''\in R^!(w)$ s.t.\ $\mathfrak{M},w'\nVdash^+\phi$ and $\mathfrak{M},w'\nVdash^-\phi$ but $\mathfrak{M},w''\nVdash^+\phi$ and $\mathfrak{M},w''\Vdash^-\phi$.
\end{enumerate}
It is also clear that either $\mathfrak{M},u\Vdash^+\chi$ and $\mathfrak{M},u\Vdash^-\chi$, or $\mathfrak{M},u\nVdash^+\chi$ and $\mathfrak{M},u\Vdash^-\chi$ in all $u\in R^!(w)$.

Again, by an application of Proposition~\ref{prop:dualvaluations}, we have that $\mathfrak{M},u\nVdash^+\chi$ and $\mathfrak{M},u\Vdash^-\chi$ in all $u\in R^!(w)$ which gives $\mathfrak{M}_\partial,w'\Vdash^+\phi$ but $\mathfrak{M}_\partial,w'\nVdash^+\chi$ (for the case $(2.a)$); and $\mathfrak{M}_\partial,w''\Vdash^+\phi$ but $\mathfrak{M}_\partial,w''\nVdash^+\chi$ (for $(2.b)$--$(2.d)$). The result follows.
\end{proof}
\section{Analytic cut system\label{sec:tableaux}}
When it comes to providing a~calculus for $\BD$ or one of its relatives or expansions, there are, usually, two avenues. The first is to provide a~Hilbert-style axiomatisation. This was done in~\cite{Dunn1995} and~\cite{Drobyshevich2020}\footnote{The completeness proof in Dunn's paper contained a~mistake that was addressed in~\cite{Drobyshevich2020}.} for $\KBD$ ($\KFDE$). However, the completeness proofs of such systems can require the introduction of non-normal worlds (cf.~\cite[\S4]{Drobyshevich2020} for a~detailed discussion). The other option is to construct a~tableaux (or analytic cut\footnote{The main difference between tableaux and analytic cut calculi is the presence of the eponymous rule in the latter. In classical logic, this rule internalises the principle of excluded middle and is formulated as $\dfrac{}{\phi\mid\neg\phi}$ with $\phi$ being a~subformula of a~formula occurring on the branch. In general, if the $\{\mathbf{v}_1,\ldots,\mathbf{v}_n\}$ is the set of truth values of a~given logic, the analytic cut rule can be given, for example, as follows: $\dfrac{}{\mathbf{v}_1[\phi]\mid\ldots\mid\mathbf{v}_n[\phi]}$. Our analytic cut rule ($\mathfrak{v}\overline{\mathfrak{v}}$ in Fig.~\ref{fig:Tknowignorules}) is an adaptation of the analytic cut rule of the $\mathbf{RE}_\mathrm{fde}$ calculus from~\cite{DAgostino1990}.}) calculus. This was done, e.g., by Priest in~\cite{Priest2008,Priest2008FromIftoIs} for $\KBD$. Similarly, in~\cite{KozhemiachenkoVashentseva2023}, we presented an analytic cut system for the expansion of $\BD$ with~$\blacktriangle$. The soundness and completeness of tableaux and analytic cut systems are usually straightforward to establish. Moreover, they can be easily expanded to accommodate new connectives and operators which is not trivial when one deals with a~Hilbert calculus.

Thus, in this section, we provide a~unified analytic cut calculus for $\knowBD$ and $\ignoBD$ that is built similarly to the calculus for the expansion of $\BD$ with $\blacktriangle$ from~\cite{KozhemiachenkoVashentseva2023} and augments D'Agostino's analytic cut calculus $\mathbf{RE}_{\mathrm{fde}}$~\cite{DAgostino1990} with additional modal rules. We are using labelled formulas whose labels contain two parts: the value of the formula and the state where the formula has this value.
\begin{definition}[Labelled formulas]\label{def:labelledformulas}
We fix a~countable set of state-labels $\mathsf{Lab}=\{w,w_0,w',\ldots\}$ and the~set of value-labels $\mathsf{Val}=\{\ltrue,\lfalse,\lnontrue,\lnonfalse\}$. A~\emph{labelled formula} is a~construction of the form $\mathsf{w}:\phi;\mathfrak{v}$ with $\phi\in\Ligno\cup\Lknow$, $\mathsf{w}\in\mathsf{Lab}$, and $\mathfrak{v}\in\mathsf{Val}$.
\end{definition}

The interpretations of labelled formulas are summarised in Table~\ref{table:labelsintepretation}.
\begin{table}
\centering
\begin{tabular}{c|c}
\textbf{Labelled formula}&\textbf{Interpretation}\\\hline
$w:\phi;\ltrue$&$\mathfrak{M},w\Vdash^+\phi$\\
$w:\phi;\lfalse$&$\mathfrak{M},w\Vdash^-\phi$\\
$w:\phi;\lnontrue$&$\mathfrak{M},w\nVdash^+\phi$\\
$w:\phi;\lnonfalse$&$\mathfrak{M},w\nVdash^-\phi$
\end{tabular}
\caption{Interpretations of labelled formulas in $\Tknowigno$ proofs.}
\label{table:labelsintepretation}
\end{table}
\begin{convention}\label{conv:conjugatesandinverses}
We set
\begin{align*}
\ltrue^\neg&=\lfalse&\lfalse^\neg&=\ltrue&
\lnontrue^\neg&=\lnonfalse&\lnonfalse^\neg&=\lnontrue
\end{align*}
\end{convention}
\begin{definition}[$\Tknowigno$ --- analytic cut for $\knowBD$ and $\ignoBD$]\label{def:tableaux}
We define an~$\Tknowigno$-proof as a~downward branching tree whose nodes are labelled with sets containing labelled formulas and constructions of the form $w\mathsf{R}w'$. Each branch can be extended by one of the rules from Fig.~\ref{fig:Tknowignorules}. A~branch $\mathcal{B}$ is \emph{closed} iff $w_i\!:\!\phi;\mathfrak{v},w_i\!:\!\phi;\overline{\mathfrak{v}}\in\mathcal{B}$ for some $\phi\in\Lknow\cup\Ligno$, $w_i\in\mathsf{Lab}$, and $\mathfrak{v}\in\mathsf{Val}$. Otherwise, $\mathcal{B}$ is \emph{open}. An open branch $\mathcal{B}$ is \emph{complete} iff the following condition is met.
\begin{itemize}
\item[$*$] If all premises of a~rule occur on the branch, then at least one conclusion of that rule occurs on the branch as well.
\end{itemize}
A tree is closed iff every branch is closed. Finally, we say that \emph{$\phi\vdash\chi$ is proved in $\Tknowigno$} iff there is a~closed tree whose root is $\{w\!:\!\phi;\ltrue,~w\!:\!\chi;\lnontrue\}$.
\end{definition}
\begin{figure}
\begin{align*}
\mathfrak{v}\overline{\mathfrak{v}}:\dfrac{}{w:\phi;\mathfrak{v}\mid w:\phi;\overline{\mathfrak{v}}}~\left(\parbox{15em}{$\phi$ is a~subformula of a~formula occurring on the branch; $w$ occurs on the branch}\right)
\end{align*}
\begin{align*}
\neg_\mathfrak{v}:\dfrac{w\!:\!\neg\phi;\mathfrak{v}}{w\!:\!\phi;\mathfrak{v}^\neg}&&\wedge_\ltrue:\dfrac{w\!:\!\phi\wedge\chi;\ltrue}{\begin{matrix}w\!:\!\phi;\ltrue\\w\!:\!\chi;\ltrue\end{matrix}}&&\vee_\ltrue:\dfrac{\begin{matrix}w:\phi_1\vee\phi_2;\ltrue\\w:\phi_i;\lnontrue\end{matrix}}{w:\phi_j;\ltrue}&&\wedge_{\lnonfalse}:\dfrac{w\!:\!\phi\wedge\chi;\lnonfalse}{\begin{matrix}w\!:\!\phi;\lnonfalse\\w\!:\!\chi;\lnonfalse\end{matrix}}&&\vee_{\lnonfalse}:\dfrac{\begin{matrix}w:\phi_1\vee\phi_2;\lnonfalse\\w:\phi_i;\lfalse\end{matrix}}{w:\phi_j;\lnonfalse}\\
&&\wedge_\lfalse:\dfrac{\begin{matrix}w:\phi_1\wedge\phi_2;\lfalse\\w:\phi_i;\lnonfalse\end{matrix}}{w:\phi_j;\lfalse}&&\vee_\lfalse:\dfrac{w\!:\!\phi\vee\chi;\lfalse}{\begin{matrix}w\!:\!\phi;\lfalse\\w\!:\!\chi;\lfalse\end{matrix}}&&\wedge_{\lnontrue}:\dfrac{\begin{matrix}w:\phi_1\wedge\phi_2;\lnontrue\\w:\phi_i;\ltrue;\end{matrix}}{w:\phi_j;\lnontrue}&&\vee_{\lnontrue}:\dfrac{w\!:\!\phi\vee\chi;\lnontrue}{\begin{matrix}w\!:\!\phi;\lnontrue\\w\!:\!\chi;\lnontrue\end{matrix}}
\end{align*}
\begin{align*}
\blacksquare_\ltrue:\dfrac{\begin{matrix}w\!:\!\blacksquare\phi;\ltrue\\w\mathsf{R}w'\end{matrix}}{w'\!:\!\phi;\ltrue}&&\blacksquare^-_\ltrue:\dfrac{\begin{matrix}w\!:\!\blacksquare\phi;\ltrue&w'\!:\!\phi;\lfalse\\w\mathsf{R}w'&w\mathsf{R}w''\end{matrix}}{w''\!:\!\phi;\lfalse}&&\blacksquare_\lfalse:\dfrac{w\!:\!\blacksquare\phi;\lfalse}{\left.\begin{matrix}w\mathsf{R}u\\u\!:\!\phi;\lfalse\end{matrix}\right|\begin{matrix}w\mathsf{R}u&w\mathsf{R}u'\\u\!:\!\phi;\ltrue&u'\!:\!\phi;\lnontrue\end{matrix}}\\
\blacksquare_{\lnonfalse}:\dfrac{\begin{matrix}w\!:\!\blacksquare\phi;\lnonfalse\\w\mathsf{R}w'\end{matrix}}{w'\!:\!\phi;\lnonfalse}&&\blacksquare^+_{\lnonfalse}:\dfrac{\begin{matrix}w\!:\!\blacksquare\phi;\lnontrue&w'\!:\!\phi;\ltrue\\w\mathsf{R}w'&w\mathsf{R}w''\end{matrix}}{w''\!:\!\phi;\ltrue}&&\blacksquare_{\lnontrue}:\dfrac{w\!:\!\blacksquare\phi;\lnontrue}{\left.\begin{matrix}w\mathsf{R}u\\u\!:\!\phi;\lnontrue\end{matrix}\right|\begin{matrix}w\mathsf{R}u&w\mathsf{R}u'\\u\!:\!\phi;\lfalse&u'\!:\!\phi;\lnonfalse\end{matrix}}
\end{align*}
\begin{align*}
\mathbf{I}_\ltrue:\dfrac{w\!:\!\mathbf{I}\phi;\ltrue}{w\!:\!\phi;\ltrue}&&\mathbf{I}^\mathsf{R}_\ltrue:\dfrac{\begin{matrix}w\!:\!\mathbf{I}\phi;\ltrue\\w\mathsf{R}s\end{matrix}}{s\!:\!\phi;\lfalse}&&\mathbf{I}^+_\ltrue:\dfrac{\begin{matrix}w\!:\!\mathbf{I}\phi;\ltrue&s\!:\!\phi;\ltrue\\w\mathsf{R}s&w\mathsf{R}s'\end{matrix}}{s'\!:\!\phi;\ltrue}&&\mathbf{I}_\lfalse:\dfrac{w\!:\!\mathbf{I}\phi;\lfalse}{\left.\begin{matrix}w\mathsf{R}u\\u\!:\!\phi;\ltrue\end{matrix}\right|\left.\begin{matrix}w\mathsf{R}u&w\mathsf{R}u'\\u\!:\!\phi;\lfalse&u'\!:\!\phi;\lnonfalse\end{matrix}\right|w\!:\!\phi;\lfalse}\\
\mathbf{I}_{\lnonfalse}:\dfrac{w\!:\!\mathbf{I}\phi;\lnonfalse}{w\!:\!\phi;\lnonfalse}&&\mathbf{I}^\mathsf{R}_\ltrue:\dfrac{\begin{matrix}w\!:\!\mathbf{I}\phi;\lnonfalse\\w\mathsf{R}s\end{matrix}}{s\!:\!\phi;\lnontrue}&&\mathbf{I}^+_{\lnonfalse}:\dfrac{\begin{matrix}w\!:\!\mathbf{I}\phi;\lnonfalse&s\!:\!\phi;\lnonfalse\\w\mathsf{R}s&w\mathsf{R}s'\end{matrix}}{s'\!:\!\phi;\lnonfalse}&&\mathbf{I}_{\lnontrue}:\dfrac{w\!:\!\mathbf{I}\phi;\lnontrue}{\left.\begin{matrix}w\mathsf{R}u\\u\!:\!\phi;\lnonfalse\end{matrix}\right|\left.\begin{matrix}w\mathsf{R}u&w\mathsf{R}u'\\u\!:\!\phi;\ltrue&u'\!:\!\phi;\lnontrue\end{matrix}\right|w\!:\!\phi;\lnontrue}
\end{align*}
\caption{$\Tknowigno$ rules: vertical bars denote branching; $\mathfrak{v}\in\mathsf{Val}$, $u$ and $u'$ are fresh on the branch, $w\neq s$, $w\neq s'$, $i\neq j$, $i,j\in\{1,2\}$.}
\label{fig:Tknowignorules}
\end{figure}

Before proceeding to the proof of soundness and completeness of $\Tknowigno$, let us consider two examples of tableaux proofs (Fig.~\ref{fig:failedproof}). Namely, we prove $\mathbf{I}p\wedge\mathbf{I}q\vdash\mathbf{I}(p\vee q)$ (recall~\eqref{equ:ignoaxioms}) and disprove $\blacksquare(p\wedge q)\vdash\blacksquare p$ (cf.~Remark~\ref{rem:knownonstandard}). For the sake of brevity, we do not apply the $\mathfrak{v}\overline{\mathfrak{v}}$ rule at $w_0$ as it is clear that these applications will not make the open branch closed. Note that the step \textbf{higlighed in boldface} is obtained from $w_0:\blacksquare(p\wedge q);\ltrue$, $w_1:p\wedge q;\lfalse$, $w_0\mathsf{R}w_1$, and $w_0\mathsf{R}w_2$ by $\blacksquare^-_\ltrue$ rule. To extract the counter-model from a~complete open branch $\mathcal{B}$, we set $W=\{w:w\text{ occurs on }\mathcal{B}\}$, set $wRw'$ iff $w\mathsf{R}w'\in\mathcal{B}$, and define valuations $v^+$ and $v^-$ according to Table~\ref{table:labelsintepretation}.
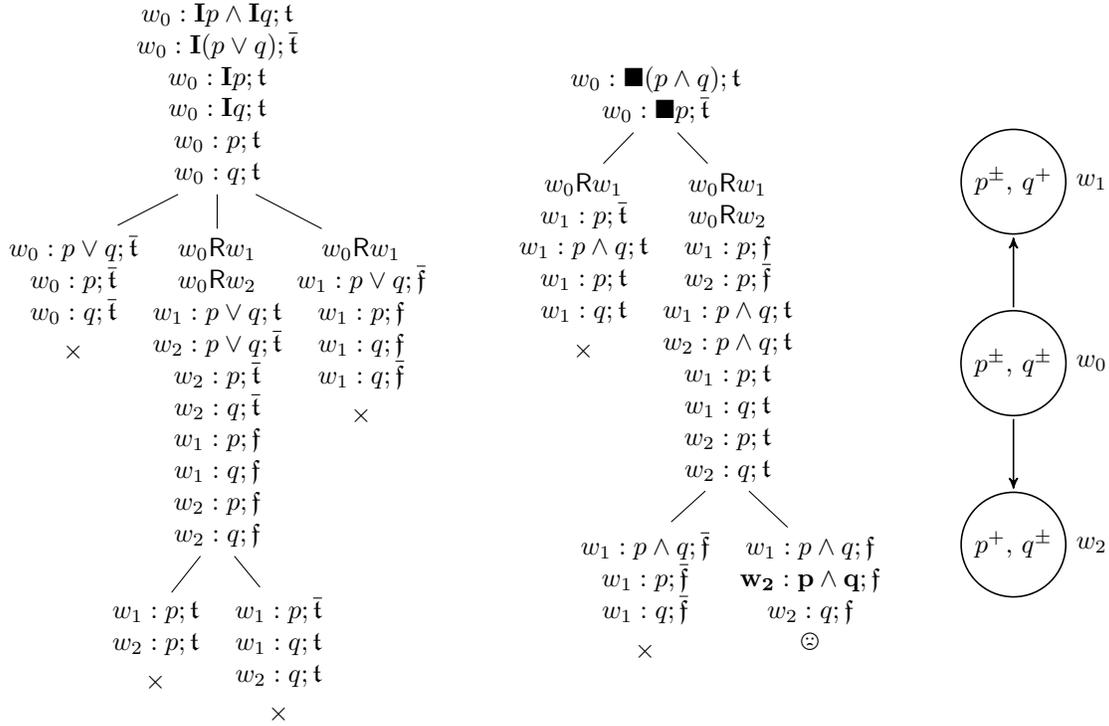
\begin{figure}
\centering
\begin{minipage}{.4\textwidth}
\centering
\begin{forest}
smullyan tableaux
[w_0:\mathbf{I}p\wedge\mathbf{I}q;\ltrue
[w_0:\mathbf{I}(p\vee q);\lnontrue
[w_0:\mathbf{I}p;\ltrue
[w_0:\mathbf{I}q;\ltrue
[w_0:p;\ltrue[w_0:q;\ltrue
[w_0:p\vee q;\lnontrue[w_0:p;\lnontrue[w_0:q;\lnontrue,closed]]]
[w_0\mathsf{R}w_1[w_0\mathsf{R}w_2[w_1:p\vee q;\ltrue[w_2:p\vee q;\lnontrue[w_2:p;\lnontrue[w_2:q;\lnontrue[w_1:p;\lfalse[w_1:q;\lfalse[w_2:p;\lfalse[w_2:q;\lfalse
[w_1:p;\ltrue[w_2:p;\ltrue,closed]][w_1:p;\lnontrue[w_1:q;\ltrue[w_2:q;\ltrue,closed]]]
]]]]]]]]]]
[w_0\mathsf{R}w_1[w_1:p\vee q;\lnonfalse[w_1:p;\lfalse[w_1:q;\lfalse[w_1:q;\lnonfalse,closed]]]]]
]]]]]]
\end{forest}
\end{minipage}
\hfill
\begin{minipage}{.3\textwidth}
\centering
\begin{forest}
smullyan tableaux
[w_0:\blacksquare(p\wedge q);\ltrue
[w_0:\blacksquare p;\lnontrue
[w_0\mathsf{R}w_1[w_1:p;\lnontrue[w_1:p\wedge q;\ltrue[w_1:p;\ltrue[w_1:q;\ltrue,closed]]]]]
[w_0\mathsf{R}w_1
[w_0\mathsf{R}w_2
[w_1:p;\lfalse
[w_2:p;\lnonfalse
[w_1:p\wedge q;\ltrue[w_2:p\wedge q;\ltrue
[w_1:p;\ltrue[w_1:q;\ltrue[w_2:p;\ltrue[w_2:q;\ltrue
[w_1:p\wedge q;\lnonfalse[w_1:p;\lnonfalse[w_1:q;\lnonfalse,closed]]][w_1:p\wedge q;\lfalse[\mathbf{w_2:p\wedge q;\lfalse}[w_2:q;\lfalse[\frownie]]]]
]]]]]]]]]]]]
\end{forest}
\end{minipage}
\hfill
\begin{minipage}{.2\textwidth}
\centering
\begin{tikzpicture}[modal,node distance=1cm,world/.append style={minimum
size=1cm}]
\node[world] (w0) [label=right:{$w_0$}] {$p^\pm$, $q^\pm$};
\node[world] (w1) [above=of w0] [label=right:{$w_1$}] {$p^\pm$, $q^+$};
\node[world] (w2) [below=of w0] [label=right:{$w_2$}] {$p^+$, $q^\pm$};
\path[->] (w0) edge (w1);
\path[->] (w0) edge (w2);
\end{tikzpicture}
\end{minipage}
\caption{An $\Tknowigno$ proof of $\mathbf{I}p\wedge\mathbf{I}q\vdash\mathbf{I}(p\vee q)$ (left); a~failed proof of $\blacksquare(p\wedge q)\vdash\blacksquare p$ (center, $\frownie$~denotes the complete open branch) and its corresponding model (right).}
\label{fig:failedproof}
\end{figure}

We are now ready to state and prove that $\Tknowigno$ is sound and complete w.r.t.\ the semantics in Definition~\ref{def:ignoknowsemantics}. Our proof is a~straightforward adaptation of~\cite[Theorems~3.7 and~3.8]{KozhemiachenkoVashentseva2023}.
\begin{definition}[Branch realisation]\label{def:branchrealisation}
We say that $\mathfrak{M}=\langle W,R,v^+,v^-\rangle$ with $W=\{w:w\text{ occurs on }\mathcal{B}\}$, $R=\{\langle w_i,w_j\rangle:w_i\mathsf{R}w_j\in\mathcal{B}\}$, and $w\in v^+(p)$ ($w\in v^-(p)$) iff $w\!:\!p;\ltrue\in\mathcal{B}$ ($w\!:\!p;\lfalse\in\mathcal{B}$) realises a~branch $\mathcal{B}$ of a~tree iff the following conditions are met.
\begin{enumerate}
\item If $w\!:\!\phi;\ltrue\in\mathcal{B}$ ($w\!:\!\phi;\lfalse\in\mathcal{B}$), then $\mathfrak{M},w\Vdash^+\phi$ ($\mathfrak{M},w\Vdash^-\phi$, respectively).
\item If $w\!:\!\phi;\lnontrue\in\mathcal{B}$ ($w\!:\!\phi;\lnonfalse\in\mathcal{B}$), then $\mathfrak{M},w\nVdash^+\phi$ ($\mathfrak{M},w\nVdash^-\phi$, respectively).
\end{enumerate}
\end{definition}
\begin{theorem}[Soundness and completeness of $\Tknowigno$]\label{theorem:Tknowignocompleteness}
For every $\phi\vdash\chi$ s.t.\ $\phi,\chi\in\Lknow\cup\Ligno$, it holds that $\phi\vdash\chi$ is valid iff it has a~$\Tknowigno$ proof.
\end{theorem}
\begin{proof}
The proof is in the appendix (Section~\ref{sec:completenessproof}).
\end{proof}
\section{Expressivity\label{sec:expressivity}}
In Section~\ref{sec:preliminaries}, we argued that the $\Box$ modality as it is often defined in $\BD$ (recall Definition~\ref{def:KBD}) is not well-suited for the formalisation of belief, knowledge, or ignorance in the Belnap--Dunn logic. In this section, we show that $\LBox$ formulas, actually, cannot formalise our interpretation of $\blacksquare$ and $\mathbf{I}$ because neither of these modalities can be defined via $\Box$. In addition, we also show that neither $\blacksquare$, nor $\mathbf{I}$ can define $\Box$ and that $\mathbf{I}$ and $\blacksquare$ are not interdefinable either. This last property of $\mathbf{I}$ and $\blacksquare$ corresponds to a~desideratum in~\cite{KubyshkinaPetrolo2021} stating that knowledge and ignorance should be independent notions.
\begin{definition}\label{def:formuladefinability}
Let $\mathcal{L}_1$ and $\mathcal{L}_2$ be two languages and let $\mathbb{K}$ be a~class of frames. We say that $\phi\in\mathcal{L}_1$ \emph{defines} $\chi\in\mathcal{L}_2$ \emph{in $\mathbb{K}$} iff for any $\mathfrak{F}\in\mathbb{K}$ and for any pointed model $\langle\mathfrak{M},w\rangle$ on $\mathfrak{F}$, it holds that
\begin{align*}
\mathfrak{M},w\Vdash^+\phi&\text{ iff }\mathfrak{M},w\Vdash^+\chi&
\mathfrak{M},w\Vdash^-\phi&\text{ iff }\mathfrak{M},w\Vdash^-\chi
\end{align*}
\end{definition}
\begin{theorem}\label{theorem:ignoknowundefinable}~
\begin{enumerate}
\item No $\LBox$ formula can define $\blacksquare p$ on the classes of all frames, all reflexive frames, all transitive frames, all symmetric frames, and all Euclidean frames.
\item No $\LBox$ formula can define $\mathbf{I}p$ on the classes of all frames, all reflexive frames, all transitive frames, all symmetric frames, and all Euclidean frames.
\end{enumerate}
\end{theorem}
\begin{proof}
Since $\blacksquare$ can define $\blacktriangle$ (Theorem~\ref{theorem:knowisgood}) and since $\Box$ cannot define $\blacktriangle$ on $\mathbf{S5}$ frames as shown in~\cite[Theorem~4.3]{KozhemiachenkoVashentseva2023}, the first part follows immediately. Let us now prove the second part.

For this, we borrow the approach from~\cite[\S3]{KubyshkinaPetrolo2021} and consider two models in Fig.~\ref{fig:ignoundefinable}. It is clear that the accessibility relations in these models are, in fact, equivalence relations and that $\mathfrak{M},w_0\Vdash^+\mathbf{I}p$ but $\mathfrak{M},w_0\nVdash^+\mathbf{I}p$. However, we can show by induction that
\begin{itemize}
\item for every $\phi\in\LBox$ s.t.\ $\mathfrak{M},w_0\Vdash^+\phi$, it holds that $\mathfrak{M}',w'_0\Vdash^+\phi$ and $\mathfrak{M}',w'_1\Vdash^+\phi$, and
\item for every $\phi\in\LBox$ s.t.\ $\mathfrak{M},w_0\Vdash^-\phi$, it holds that $\mathfrak{M}',w'_0\Vdash^-\phi$ $\mathfrak{M}',w'_1\Vdash^-\phi$.
\end{itemize}

The basis case of variables holds by the construction of $\mathfrak{M}$ and $\mathfrak{M}'$. The cases of propositional connectives can be established by induction hypothesis. We consider $\phi=\Box\chi$ and let $\mathfrak{M},w_0\Vdash^+\Box\chi$. Then, $\mathfrak{M},w_0\Vdash^+\chi$. By the induction hypothesis, we have $\mathfrak{M}',w'_0\Vdash^+\chi$ and $\mathfrak{M}',w'_1\Vdash^+\chi$, whence, $\mathfrak{M}',w'_0\Vdash^+\Box\chi$, as required. Now let $\mathfrak{M},w_0\Vdash^-\Box\chi$. Hence, $\mathfrak{M},w_0\Vdash^-\chi$. By the induction hypothesis, we have $\mathfrak{M}',w'_0\Vdash^-\chi$ and $\mathfrak{M}',w'_1\Vdash^-\chi$, whence, $\mathfrak{M}',w'_0\Vdash^-\Box\chi$, as required.

The result follows.
\begin{figure}
\centering
\begin{tikzpicture}[modal,node distance=1cm,world/.append style={minimum
size=1cm}]
\node[world] (w0) [label=below:{$w_0$}] {$p^+$};
\node[] [left=of w0] {$\mathfrak{M}$:};
\path[->] (w0) edge[reflexive] (w0);
\end{tikzpicture}
\hfil
\begin{tikzpicture}[modal,node distance=1cm,world/.append style={minimum
size=1cm}]
\node[world] (w0) [label=below:{$w'_0$}] {$p^+$};
\node[world] (w1) [right=of w0] [label=below:{$w'_1$}] {$p^+$};
\node[] [left=of w0] {$\mathfrak{M}'$:};
\path[->] (w0) edge[reflexive] (w0);
\path[->] (w1) edge[reflexive] (w1);
\path[<->] (w0) edge (w1);
\end{tikzpicture}
\caption{All variables have the same value exemplified by $p$.}
\label{fig:ignoundefinable}
\end{figure}
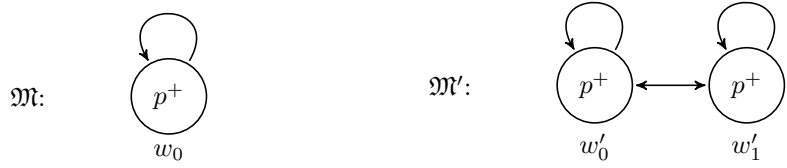
\end{proof}

\begin{theorem}\label{theorem:Boxundefinable}
There is no formula $\phi\in\Lknow\cup\Ligno$ that defines $\Box p$ on the classes of all frames, all reflexive frames, all transitive frames, all symmetric frames, and all Euclidean frames.
\end{theorem}
\begin{proof}
We consider Fig.~\ref{fig:Boxundefinable}. It is clear that its accessibility relation is an equivalence relation and that $\mathfrak{M},w_0\Vdash^+\Box p$ and $\mathfrak{M},w_0\Vdash^-\Box p$ as well as $\mathfrak{M},w_2\Vdash^+\Box p$ and $\mathfrak{M},w_2\Vdash^-\Box p$. We show that there is no $\phi\in\Lknow\cup\Ligno$ s.t.\
\begin{enumerate}
\item $\mathfrak{M},w_0\Vdash^+\phi$ and $\mathfrak{M},w_0\Vdash^-\phi$, and
\item $\mathfrak{M},w_2\Vdash^+\phi$ and $\mathfrak{M},w_2\Vdash^-\phi$.
\end{enumerate}
We proceed by induction on $\phi$ and reason for a~contradiction.

If $\phi=p$, it is clear that $\mathfrak{M},w_0\nVdash^-p$ and $\mathfrak{M},w_2\nVdash^-p$. If $\phi=\neg\chi$, assume that $\mathfrak{M},w_0\Vdash^+\neg\chi$ and $\mathfrak{M},w_0\Vdash^-\neg\chi$. Then, $\mathfrak{M},w_0\Vdash^+\chi$ and $\mathfrak{M},w_0\Vdash^-\chi$. A~contradiction. If $\phi=\chi\wedge\psi$, again, assume that $\mathfrak{M},w_0\Vdash^+\chi\wedge\psi$ and $\mathfrak{M},w_0\Vdash^-\chi\wedge\psi$. Then, $\mathfrak{M},w_0\Vdash^+\chi$, $\mathfrak{M},w_0\Vdash^+\psi$ and either $\mathfrak{M},w_0\Vdash^-\chi$ or $\mathfrak{M},w_0\Vdash^-\psi$. In both cases, we have a~contradiction. The case of $\phi=\chi\vee\psi$ can be proven in the same fashion.

Let $\phi=\blacksquare\chi$. If $\mathfrak{M},w_0\Vdash^+\blacksquare\chi$ and $\mathfrak{M},w_0\Vdash^-\blacksquare\chi$ as well as $\mathfrak{M},w_2\Vdash^+\blacksquare\chi$ and $\mathfrak{M},w_2\Vdash^-\blacksquare\chi$, then $\mathfrak{M},w_0\Vdash^+\chi$ and $\mathfrak{M},w_0\Vdash^-\chi$ (and $\mathfrak{M},w_2\Vdash^+\chi$ and $\mathfrak{M},w_2\Vdash^-\chi$, as well). But by the inductive hypothesis, there cannot be $\chi\in\Lknow\cup\Ligno$ s.t.\ $\mathfrak{M},w_0\Vdash^+\chi$, $\mathfrak{M},w_0\Vdash^-\chi$, $\mathfrak{M},w_2\Vdash^+\chi$, and $\mathfrak{M},w_2\Vdash^-\chi$. Again, a~contradiction.

Finally, consider $\phi=\mathbf{I}\chi$ and assume that $\mathfrak{M},w_0\Vdash^+\mathbf{I}\chi$ and $\mathfrak{M},w_0\Vdash^-\mathbf{I}\chi$, $\mathfrak{M},w_2\Vdash^+\mathbf{I}\chi$, and $\mathfrak{M},w_2\Vdash^-\mathbf{I}\chi$. Then, we must have that $\mathfrak{M},u\Vdash^+\chi$ and $\mathfrak{M},u\Vdash^-\chi$ for all $u\in R^!(w_0)$ and also $\mathfrak{M},u'\Vdash^+\chi$ and $\mathfrak{M},u'\Vdash^-\chi$ for all $u\in R^!(w_2)$. But since $w_0\in R^!(w_2)$ and $w_2\in R^!(w_0)$, we again have a~contradiction. The result follows.
\begin{figure}
\centering
\begin{tikzpicture}[modal,node distance=0.5cm,world/.append style={minimum
size=1cm}]
\node[world] (w0) [label=below:{$w_0$}] {$p^+$};
\node[world] (w1) [right=of w0] [label=below:{$w_1$}] {$p^\pm$};
\node[world] (w2) [left=of w0] [label=below:{$w_2$}] {$p^+$};
\node[] [left=of w2] {$\mathfrak{M}$:};
\path[->] (w0) edge[reflexive] (w0);
\path[->] (w2) edge[reflexive] (w2);
\path[->] (w1) edge[reflexive] (w1);
\path[<->] (w0) edge (w1);
\path[<->] (w0) edge (w2);
\path[<->] (w1) edge[bend left=60] (w2);
\end{tikzpicture}
\caption{All variables have the same values exemplified by $p$.}
\label{fig:Boxundefinable}
\end{figure}
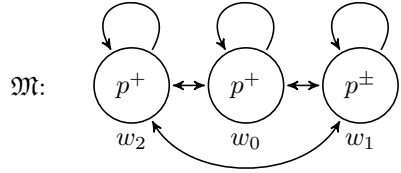
\end{proof}
\begin{theorem}\label{theorem:ignovsknow}~
\begin{enumerate}
\item No $\Lknow$ formula can define $\mathbf{I}p$ on the classes of all frames, all reflexive frames, all transitive frames, all symmetric frames, and all Euclidean frames.
\item No $\Ligno$ formula can define $\blacksquare p$ on the classes of all frames, all transitive frames, all symmetric frames, and all Euclidean frames.
\end{enumerate}
\end{theorem}
\begin{proof}
The first part can be proven in the same manner as \emph{the second part} of Theorem~\ref{theorem:ignoknowundefinable}: we use the models Fig.~\ref{fig:ignoundefinable} and show that 
\begin{itemize}
\item for every $\phi\in\Lknow$ s.t.\ $\mathfrak{M},w_0\Vdash^+\phi$, it holds that $\mathfrak{M}',w'_0\Vdash^+\phi$ and $\mathfrak{M}',w'_1\Vdash^+\phi$, and
\item for every $\phi\in\Lknow$ s.t.\ $\mathfrak{M},w_0\Vdash^-\phi$, it holds that $\mathfrak{M}',w'_0\Vdash^-\phi$ $\mathfrak{M}',w'_1\Vdash^-\phi$.
\end{itemize}

For the second part, we borrow the approach from~\cite[Observation~1.25]{GilbertKubyshkinaPetroloVenturi2022}. Namely, consider Fig.~\ref{fig:ignovsknow}. It is easy to see that $\mathfrak{M},w_0\Vdash^-\blacksquare p$ but $\mathfrak{M}',w'_0\nVdash^-\blacksquare p$. On the other hand, one can check by induction on $\phi\in\Ligno$ that $\mathfrak{M},w_0\Vdash^+\phi$ iff $\mathfrak{M}',w'_0\Vdash^+\phi$ and $\mathfrak{M},w_0\Vdash^-\phi$ iff $\mathfrak{M}',w'_0\Vdash^-\phi$.
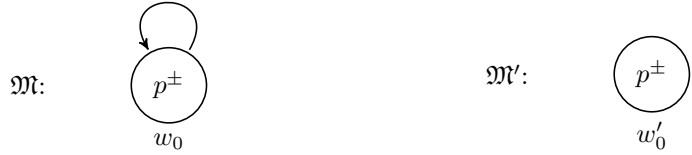
\begin{figure}
\centering
\begin{tikzpicture}[modal,node distance=1cm,world/.append style={minimum
size=1cm}]
\node[world] (w0) [label=below:{$w_0$}] {$p^\pm$};
\node[] [left=of w0] {$\mathfrak{M}$:};
\path[->] (w0) edge[reflexive] (w0);
\end{tikzpicture}
\hfil
\begin{tikzpicture}[modal,node distance=1cm,world/.append style={minimum
size=1cm}]
\node[world] (w0) [label=below:{$w'_0$}] {$p^\pm$};
\node[] [left=of w0] {$\mathfrak{M}'$:};
\end{tikzpicture}
\caption{All variables in both models have the same values exemplified by $p$.}
\label{fig:ignovsknow}
\end{figure}
\end{proof}
\begin{remark}\label{rem:Boxundefinable}
Note from the proof of Theorem~\ref{theorem:Boxundefinable} that not only $\blacksquare$ and $\mathbf{I}$ \emph{by themselves} cannot define~$\Box$ but even \emph{the language that combines both $\blacksquare$ and $\mathbf{I}$} cannot define $\Box$.
\end{remark}

We finish the section with a~brief discussion of the unknown truth operator $\bullet$. Recall from Definition~\ref{def:ignoknowsemantics} that $\bullet\phi$ is defined as $\phi\wedge\neg\blacksquare\phi$. The next statement shows that it cannot be defined using~$\Box$.
\begin{theorem}\label{theorem:accidenceundefinable}
There is no formula $\phi\in\LBox$ that defines $\bullet p$ on the classes of all frames, all reflexive frames, all transitive frames, all symmetric frames, and all Euclidean frames.
\end{theorem}
\begin{proof}
The proof is essentially the same as that of~\cite[Theorem~4.3]{KozhemiachenkoVashentseva2023} and is put in the appendix (Section~\ref{sec:accidenceproof}).
\end{proof}
\section{Frame definability\label{sec:definability}}
We have previously argued (Section~\ref{sec:preliminaries} and Theorem~\ref{theorem:knowisgood}) that $\blacksquare$ modality is well-suited to formalise knowledge and belief in the Belnap--Dunn framework: it entails “knowledge whether” and “opinionatedness” (in contrast to $\Box$); three standard axioms of epistemic modalities (truthfulness, positive introspection, and negative introspection) are valid on $\mathbf{S5}$ frames; $\mathbf{S5}$ frames are definable in $\Lknow$. In this section, we further investigate which important classes of frames are definable in $\Lknow$.

Recall the notion of frame definability.
\begin{definition}[Frame definability]
Let $\Sigma$ be a~set of sequents of the form $\phi\vdash\chi$ and $\mathbb{F}$ be a~class of frames. We say that $\Sigma$ \emph{defines} $\mathbb{F}$ iff for any frame $\mathfrak{F}$, $\mathfrak{F}\in\mathbb{F}$ iff all sequents from $\Sigma$ are valid on~$\mathfrak{F}$. A~class of frames is definable in $\Lknow$ iff there is a~set of sequents where $\phi,\chi\in\Lknow$ that defines it.
\end{definition}

Let $\mathfrak{F}=\langle W,R\rangle$ be a~frame. We will use the notation given in Table~\ref{table:classesofframes} to designate different classes of frames.
\begin{table}
\centering
\begin{tabular}{c|c}
\textbf{notation}&\textbf{class of frames}\\\hline
$\mathbf{D}$&$R$ is serial: $\forall x\exists y:R(x,y)$\\
$\mathbf{dn}$&$R$ is dense: $\forall x,y:R(x,y)\Rightarrow\exists z(R(x,z)~\&~R(z,y))$\\
$\mathbf{T}$&$R$ is reflexive\\
$\mathbf{5}$&$R$ is Euclidean: $\forall x,y,z:R(x,y)~\&~R(x,z)\Rightarrow R(y,z)$
\end{tabular}
\caption{Shorthands for the classes of frames.}
\label{table:classesofframes}
\end{table}

In the remainder of this section, we are going to show that all classes of frames given in Table~\ref{table:classesofframes} are definable. Note that the definability of $\mathbf{T}$- and $\mathbf{S5}$-frames can be obtained indirectly since they are definable using $\blacktriangle$ (which, in its turn, is definable via $\blacksquare$). However, we are going to show that the $\Lknow$-definitions of all these classes of frames are identical to their usual definitions up to the replacement of $\blacksquare$ with $\Box$, $\blacklozenge$ with $\lozenge$, and the classical implication with $\vdash$.
\begin{theorem}\label{theorem:definability}
All classes of frames given in Table~\ref{table:classesofframes} are definable in $\Lknow$.
\end{theorem}
\begin{proof}~

\fbox{$\mathbf{D}$ frames}\\
We show that $\mathfrak{F}\in\mathbf{D}$ iff $\mathfrak{F}\models[\blacksquare p\vdash\blacklozenge p]$. Let $\mathfrak{F}$ be serial, and $\mathfrak{M}$ be a~model on $\mathfrak{F}$ s.t.\ $\mathfrak{M},w\Vdash^+\blacksquare p$. We show that $\mathfrak{M},w\Vdash^+\blacklozenge p$. Since $\mathfrak{F}$ is serial, we have that there is some $w'\in R(w)$. Moreover, in every $w''\in R(w)$, it holds that $\mathfrak{M},w''\Vdash^+p$. Thus, $\mathfrak{M},w\Vdash^+\blacklozenge p$, as required.

For the converse, let $\mathfrak{F}\notin\mathbf{D}$ and $w\in\mathfrak{F}$ be s.t.\ $R(w)=\varnothing$. It is clear that for every model $\mathfrak{M}$ on~$\mathfrak{F}$, we have $\mathfrak{M},w\Vdash^+\blacksquare p$ but $\mathfrak{M},w\nVdash^+\blacklozenge p$.

\fbox{$\mathbf{T}$ frames}\\
We show that $\mathfrak{F}\in\mathbf{T}$ iff $\mathfrak{F}\models[\blacksquare p\vdash p]$. Let $\mathfrak{F}$ be reflexive and $\mathfrak{M}$ be a~model on $\mathfrak{M}$ s.t.\ $\mathfrak{M},w\Vdash^+\blacksquare p$. Since $w\in R(w)$, it is clear that $\mathfrak{M},w\Vdash^+p$, as required. For the converse, let $\mathfrak{F}\notin\mathbf{T}$ and $w\in\mathfrak{F}$ be s.t.\ $w\notin R(w)$. Now, for every $w'\in R(w)$, we set $w'\in v^+(p)$ and $w'\notin v^-(p)$. For $w$, we define $w\notin v^+(p)$ and $w\notin v^-(p)$. It is clear that $\mathfrak{M},w\Vdash^+\blacksquare p$ but $\mathfrak{M},w\nVdash^+p$.

\fbox{$\mathbf{dn}$ frames}\\
We show that $\mathfrak{F}\in\mathbf{dn}$ iff $\mathfrak{F}\models[\blacklozenge p\vdash\blacklozenge\blacklozenge p]$. Let $\mathfrak{F}$ be dense and $\mathfrak{M},w\Vdash^+\blacklozenge p$ for some model $\mathfrak{M}$ on $\mathfrak{F}$. We have two cases:
\begin{enumerate}
\item there is $w'\in R(w)$ s.t.\ $\mathfrak{M},w'\Vdash^+p$, or
\item there are $u,u'\in R(w)$ s.t.\ $\mathfrak{M},u\Vdash^-p$ but $\mathfrak{M},u'\nVdash^-p$.
\end{enumerate}
In the first case, there is a~state $w''$ s.t.\ $w''Rw'$ and $wRw''$. Thus, $\mathfrak{M},w''\!\Vdash^+\!\blacklozenge p$, whence $\mathfrak{M},w\Vdash^+\blacklozenge\blacklozenge p$, as required. In the second case, there are $t$ and $t'$ s.t.\ $wRtRu$ and $wRt'Ru'$, whence $\mathfrak{M},t'\nVdash^-\blacklozenge p$. It now suffices to show that $\mathfrak{M},t\Vdash^+\blacklozenge p$ or $\mathfrak{M},t\Vdash^-\blacklozenge p$.

Assume for contradiction that $\mathfrak{M},t\nVdash^+\blacklozenge p$ and $\mathfrak{M},t\nVdash^-\blacklozenge p$. Then $\mathfrak{M},s\nVdash^+p$ and $\mathfrak{M},s\nVdash^-p$ in all $s\in R(t)$ and $\mathfrak{M},u\nVdash^-p$, in particular. A~contradiction. Now, we have that $\mathfrak{M},t\Vdash^+\blacklozenge p$ or $\mathfrak{M},t\Vdash^-\blacklozenge p$ from where (since $\mathfrak{M},t'\nVdash^-\blacklozenge p$) we obtain that $\mathfrak{M},w\Vdash^+\blacklozenge\blacklozenge p$, as required.

For the converse, let $\mathfrak{F}$ be not dense and $w,w'\in\mathfrak{F}$ be s.t.\ $wRw'$ but for no $wRu$ and $uRw'$. We set the valuations as follows: $w'\in v^+(p)$ and $w'\notin v^-(p)$; $t\notin v^+(p)$ and $t\in v^-(p)$ for all $t\neq w'$. It is clear that $\mathfrak{M},w\Vdash^+\blacklozenge p$ but $\mathfrak{M},s\nVdash^+\blacklozenge p$ and $\mathfrak{M},s\Vdash^-\blacklozenge p$ for all $s\in R(w)$. Thus, $\mathfrak{M},w\nVdash^+\blacklozenge\blacklozenge p$, as required.

\fbox{$\mathbf{5}$ frames}\\
Observe from the proof of Theorem~\ref{theorem:knowisgood} that if $\mathfrak{F}$ is Euclidean, then $\mathfrak{F}\models[\blacklozenge p\Vdash\blacksquare\blacklozenge p]$. We show the converse direction. Assume that $\mathfrak{F}\notin\mathbf{5}$ and that $w_0Rw_1$, $w_0Rw_2$, but $w_2\notin R(w_1)$. We set the valuation as follows: $w'\notin v^+(p)$ and $w'\in v^-(p)$ for all $w'\in R(w_1)$; $w\in v^+(p)$ and $w\notin v^-(p)$ for all other $w$'s. It is clear that $\mathfrak{M},w_0\Vdash^+\blacklozenge p$ (since $w_2\notin R(w_1)$), but $\mathfrak{M},w_1\nVdash^+\blacklozenge p$ and $\mathfrak{M},w_1\Vdash^-\blacklozenge p$. Thus, $\mathfrak{M},w_0\nVdash^+\blacksquare\blacklozenge p$, as required.
\end{proof}

From the above theorem, it follows immediately that $\mathbf{D5}$ (serial and Euclidean), $\mathbf{S5}$ (reflexive and Euclidean), and $\mathbf{Ddn}$ (serial and dense) frames are definable. Note, however, that the standard definition of transitive frames --- $\blacksquare p\vdash\blacksquare\blacksquare p$ --- does not hold (although, it is still open whether transitive frames are definable in $\Lknow$). Indeed, consider Fig.~\ref{fig:4counterexample}. The expected definition fails because, on one hand, $\blacksquare\phi$ is exactly true at $w$ when $R(w)=\varnothing$, and, on the other hand, for $\blacksquare\phi$ to be true at a~given state, $\phi$ has to have the same Belnapian value in all accessible states. These two conditions can be at odds in \emph{non-serial} transitive models as Fig.~\ref{fig:4counterexample} shows.

Still, some \emph{classes of transitive frames} are definable in the expected manner.
\begin{figure}
\centering
\begin{tikzpicture}[modal,node distance=0.5cm,world/.append style={minimum
size=1cm}]
\node[world] (w) [label=below:{$w$}] {$p^\pm$};
\node[world] (w') [right=of w] [label=below:{$w'$}] {$p^\pm$};
\node[world] (w'') [right=of w'] [label=below:{$w''$}] {$p^\pm$};
\node[] [left=of w] {$\mathfrak{M}$:};
\path[->] (w) edge (w');
\path[->] (w') edge (w'');
\path[->] (w) edge[bend left=40] (w'');
\end{tikzpicture}
\caption{$\mathfrak{M},w\Vdash^+\blacksquare p$ but $\mathfrak{M},w\nVdash^+\blacksquare\blacksquare p$.}
\label{fig:4counterexample}
\end{figure}
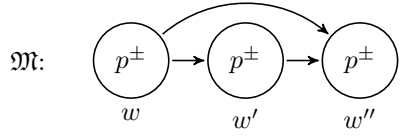
\begin{table}
\centering
\begin{tabular}{c|c}
\textbf{notation}&\textbf{class of frames}\\\hline
$\mathbf{45}$&$R$ is transitive and Euclidean\\
$\mathbf{D4}$&$R$ is serial and transitive\\
$\mathbf{S4}$&$R$ is reflexive and transitive
\end{tabular}
\caption{Shorthands for classes of transitive frames.}
\label{table:transitiveframes}
\end{table}
\begin{theorem}\label{theorem:transitivedefinability}
The classes of frames in Table~\ref{table:transitiveframes} are definable in $\Lknow$.
\end{theorem}
\begin{proof}~

\fbox{$\mathbf{45}$ frames}\\
We show that $\mathfrak{F}\in\mathbf{45}$ iff $\mathfrak{F}\models[\blacklozenge p\vdash\blacksquare\blacklozenge p]$ and $\mathfrak{F}\models[\blacksquare p\vdash\blacksquare\blacksquare p]$. Let $\mathfrak{F}$ be transitive and Euclidean. Then $\mathfrak{F}\models[\blacklozenge p\vdash\blacksquare\blacklozenge p]$ by the previous case and Theorem~\ref{theorem:knowisgood}. We show that $\mathfrak{F}\models[\blacksquare p\vdash\blacksquare\blacksquare p]$, as well. Assume that $\mathfrak{M},w\Vdash^+\blacksquare p$. Then, $\mathfrak{M},w'\Vdash^+p$ for every $w'\in R(w)$. In addition, if $\mathfrak{M},u\Vdash^-p$ in some $u\in R(w)$, then $\mathfrak{M},u'\Vdash^-p$ in all $u'\in R(w)$. Moreover, since $\mathfrak{F}$ is Euclidean, it holds that if $R(w)\neq\varnothing$, then $R(w')\neq\varnothing$ for all $w'\in R(w)$ as well. Of course, if $R(w)=\varnothing$, then $\mathfrak{M},w\Vdash^+\blacksquare\blacksquare p$. If $R(w)\neq\varnothing$, then (since $R$ is transitive) either $\mathfrak{M},w'\Vdash^+\blacksquare p$ and $\mathfrak{M},w'\nVdash^-\blacksquare p$ in every $w'\in R(w)$, or $\mathfrak{M},w'\Vdash^+\blacksquare p$ and $\mathfrak{M},w'\Vdash^-\blacksquare p$ in every $w'\in R(w)$. In both cases, $\mathfrak{M},w\Vdash^+\blacksquare\blacksquare p$, as required.

For the converse, let $\mathfrak{F}\notin\mathbf{45}$. If $\mathfrak{F}$ is not Euclidean, then $\mathfrak{F}\not\models[\blacklozenge p\vdash\blacksquare\blacklozenge p]$. So, we consider the case when $\mathfrak{F}$ is not transitive. Then, there are states $w_0$, $w_1$, and $w_2$ s.t.\ $w_0Rw_1Rw_2$ but $w_2\notin R(w_0)$. We set the valuation as follows: if $w\in R(w_0)$, then $w\in v^+(p)$ and $w\notin v^-(p)$; otherwise, $w\notin v^+(p)$ and $w\in v^-(p)$. It is clear that $\mathfrak{M},w_0\Vdash^+\blacksquare p$ but $\mathfrak{M},w_0\nVdash^+\blacksquare\blacksquare p$. Thus, $\mathfrak{F}\not\models[\blacksquare p\vdash\blacksquare\blacksquare p]$, as required.

\fbox{$\mathbf{D4}$ frames}\\
We show that $\mathfrak{F}\in\mathbf{D4}$ iff $\blacksquare p\vdash\blacklozenge p$ and $\blacksquare p\vdash\blacksquare\blacksquare p$ are valid on $\mathfrak{F}$. Let $\mathfrak{F}$ be serial and transitive. Then $\blacksquare p\vdash\blacklozenge p$ is valid on $\mathfrak{F}$ (Theorem~\ref{theorem:definability}). We show that $\blacksquare p\vdash\blacksquare\blacksquare p$ is valid as well. Let $\mathfrak{M},w\Vdash^+\blacksquare p$. Then either $\mathfrak{M},w'\Vdash^+p$ and $\mathfrak{M},w'\nVdash^-p$ in all $w'\in R(w)$ or $\mathfrak{M},w'\Vdash^+p$ and $\mathfrak{M},w'\Vdash^-p$ in all $w'\in R(w)$. In the first case, since $R$ is transitive, we have that $\mathfrak{M},w'\Vdash^+\blacksquare p$ and $\mathfrak{M},w'\nVdash^-\blacksquare p$ in all $w'\in R(w)$, whence $\mathfrak{M},w\Vdash^+\blacksquare\blacksquare p$. In the second case, observe that $R$ is not only transitive but serial as well (i.e., $R(w')\neq\varnothing$ for all $w'\in R(w)$). Thus, $\mathfrak{M},w'\Vdash^+\blacksquare p$ and $\mathfrak{M},w'\Vdash^-\blacksquare p$ in all $w'\in R(w)$, and again, $\mathfrak{M},w\Vdash^+\blacksquare\blacksquare p$, as required.

For the contrary, let $\mathfrak{F}\notin\mathbf{D4}$. If $\mathfrak{F}$ is not serial, then $\mathfrak{F}\not\models[\blacksquare p\vdash\blacklozenge p]$. If $\mathfrak{F}$ is not transitive, we can show that $\mathfrak{F}\not\models[\blacksquare p\vdash\blacksquare\blacksquare p]$ in the same way as in the previous case.

\fbox{$\mathbf{S4}$ frames}\\
We show that $\mathfrak{F}\in\mathbf{S4}$ iff $\blacksquare p\vdash p$ and $\blacksquare p\vdash\blacksquare\blacksquare p$ are valid on $\mathfrak{F}$. Again, since $\mathfrak{F}\in\mathbf{S4}$, $\blacksquare p\vdash p$ is valid on $\mathfrak{F}$ by Theorem~\ref{theorem:definability}. The validity of $\blacksquare p\vdash\blacksquare\blacksquare p$ can be checked in the same way as in the case of $\mathbf{D4}$ frames. The converse direction can also be shown in the same way as the converse direction for the case of $\mathbf{45}$ frames.
\end{proof}

Let us quickly recapitulate the results of this section. We established that the traditional epistemic ($\mathbf{S5}$) and doxastic ($\mathbf{45}$ and $\mathbf{D45}$\footnote{$\mathbf{D45}$ frames are definable since both $\mathbf{D}$ and $\mathbf{45}$ frames are definable.}) frames are definable in an expected way. In fact, $\mathbf{S4}$ can be (cf., e.g.,~\cite{Steinsvold2008}) viewed as a~logic of knowledge if one does not assume \emph{negative introspection}. In this case, $\mathbf{D4}$ can be considered a~doxastic logic (again, without the negative introspection). We have shown that both $\mathbf{D4}$ and $\mathbf{S4}$\footnote{We remind the readers again that $\mathbf{S4}$ frames are definable even with $\blacktriangle$ (knowledge whether) operator~\cite[Theorem~5.4]{KozhemiachenkoVashentseva2023}. The present result while not new is important because $\mathbf{S4}$ and $\mathbf{D4}$ frames are definable \emph{in a~standard manner}.} frames are definable.
\section{Conclusion\label{sec:conclusion}}
In this paper, we continued the line of research proposed in~\cite{KozhemiachenkoVashentseva2023} and provided expansions of the Belnap--Dunn logic (First-Degree entailment) with the knowledge ($\blacksquare$) and ignorance ($\mathbf{I}$) modalities. We presented and motivated their semantics and explored their properties and constructed a~sound and complete (Theorem~\ref{theorem:Tknowignocompleteness}) analytic cut calculus for $\knowBD$ and $\ignoBD$. Below, we summarise the main results of the paper.
\begin{itemize}
\item[$*$] The ignorance modality $\mathbf{I}$ satisfies the desiderata outlined in~\cite{KubyshkinaPetrolo2021}. Namely, the $\ignoBD$-counterparts of classical axioms and rules are $\ignoBD$ valid as shown in Theorem~\ref{theorem:ignoisgood}, and $\mathbf{I}$ can be defined neither via the standard $\Box$ operator nor via the knowledge operator $\blacksquare$ (Theorems~\ref{theorem:ignoknowundefinable} and~\ref{theorem:ignovsknow}).
\item[$*$] The introduced modality $\blacksquare$ cannot be defined via $\Box$ nor $\mathbf{I}$ (Theorems~\ref{theorem:ignoknowundefinable} and~\ref{theorem:ignovsknow}). Furthermore, it conforms to the usual requirements of knowledge modalities: $\blacksquare$ has the expected connection to the “knowledge whether” ($\blacktriangle$); truthfulness, positive, and negative introspection are valid on $\mathbf{S5}$ frames (Theorem~\ref{theorem:knowisgood}). Moreover, using $\blacksquare$, we defined an unknown truth operator $\bullet$ and showed that it is not definable via $\Box$ either (Theorem~\ref{theorem:accidenceundefinable}).
\item[$*$] Several important classes of epistemic and doxastic frames are definable using $\blacksquare$ in the same way as they are defined in classical logic up to the replacement of $\Box$ with $\blacksquare$ and the implication with~$\vdash$ (Theorems~\ref{theorem:definability} and~\ref{theorem:transitivedefinability}).
\end{itemize}

Several questions remain open. First of all, recall from Remark~\ref{rem:knownonstandard} that $\blacksquare$ behaves in a~non-standard manner. This non-standard behaviour could be rectified if we defined the validity of sequents not via the preservation of truth but via the preservation of truth and non-falsity. I.e., if we considered $\phi\vdash\chi$ to be valid when there is no such pointed model $\langle\mathfrak{M},w\rangle$ where $\phi$ is \emph{true and non-false} but $\chi$ is not. In this setting, $\blacksquare(p\wedge q)\vdash\blacksquare p\wedge\blacksquare q$ would be universally valid. This, however, would make our logic non-paraconsistent and (arguably, a~bigger issue) would render reasoning by cases~--- $\dfrac{\phi\vdash\psi\quad\chi\vdash\psi}{\phi\vee\chi\vdash\psi}$~--- unsound. In fact, such a~definition of validity would make our logic an extension of $\mathbf{ETL}$ (Exactly True Logic), and hence, we would lose the conservativity over $\BD$. $\mathbf{ETL}$ was introduced in~\cite{PietzRiveccio2013} and further studied in~\cite{ShramkoZaitsevBelikov2017,ShramkoZaitsevBelikov2018,KapsnerRivieccio2023}. We leave the analysis of $\blacksquare$ in $\mathbf{ETL}$ for future research.

Second, we have established the definability of several classes of frames in $\Lknow$. To the best of our knowledge, there are no results on the correspondence between the formulas with the classical ignorance modality $\mathbb{I}$ and classes of frames. Moreover, it is open which classes of frames \emph{are not definable} in $\Lknow$ and $\Ligno$.

Finally, many modal expansions of $\BD$ (cf., e.g.,~\cite{OdintsovWansing2010,OdintsovWansing2017}) contain an implication. Thus, considering a~counterpart of $\mathsf{BK}^\Box$\footnote{A logic expanding $\BD$ with $\Box$ and a four-valued “Nelsonian” implication.} but with $\blacksquare$ and $\mathbf{I}$ also makes sense. Moreover, adding an implication will allow for easier construction of Hilbert-style calculi and facilitate the study of the correspondence between classes of frames and formulas (both in the language with $\blacksquare$ and $\mathbf{I}$). In addition, the implication in $\mathsf{BK}^\Box$ is not the only possible one that we can add. Thus, it is also possible to compare different implicative expansions of $\ignoBD$ and $\knowBD$, especially because many important modal formulas contain multiple implications or do not have implication as their principal connective. Finally, complete axiomatisations of implicative expansions of $\knowBD$ and $\ignoBD$ will shed light on which formulas are valid on all frames.
\bibliographystyle{plain}
\bibliography{references}
\appendix
\section{Proof of Theorem~\ref{theorem:Tknowignocompleteness}\label{sec:completenessproof}}
\emph{For every $\phi\vdash\chi$ s.t.\ $\phi,\chi\in\Lknow\cup\Ligno$, it holds that $\phi\vdash\chi$ is valid iff it has a~$\Tknowigno$ proof.}
\begin{proof}
For the soundness part, one can check that if a~branch is realised by a~model, then its extension by any rule is realised too. Note also that a~closed branch clearly cannot be realised. Thus, if the tree is closed, then the initial labelled formulas are not realisable. But in order to prove $\phi\vdash\chi$, we start a~tree with $\{\phi\!:\!\ltrue;w_0,\chi\!:\!\lnontrue;w_0\}$. Hence, if this set cannot be realised, the sequent is valid.

For the completeness part, we proceed by contraposition. We show that every complete open branch $\mathcal{B}$ is realisable. Namely, we prove by induction on $\phi$ that $w\colon\phi;\ltrue\in\mathcal{B}$ iff $\mathfrak{M},w\Vdash^+\phi$ and $w\colon\phi;\lfalse\in\mathcal{B}$ iff $\mathfrak{M},w\Vdash^-\phi$ with $\mathfrak{M}$ as in Definition~\ref{def:branchrealisation}. The basis case of $\phi=p$ holds by the construction of $\mathfrak{M}$. The propositional cases are straightforward. Thus, we are going to consider only the most instructive cases of $\phi=\blacksquare\chi$ and $\phi=\mathbf{I}\chi$.

Let $w_i\colon\blacksquare\chi;\ltrue\in\mathcal{B}$. Since $\mathcal{B}$ is complete, we have that $w_j\colon\chi;\ltrue\in\mathcal{B}$ for every $w_j$ s.t.\ $w_i\mathsf{R}w_j\in\mathcal{B}$ (using $\blacksquare_\ltrue$). Moreover, if there is some $w'$ s.t.\ $w'\colon\chi;\lfalse\in\mathcal{B}$ and $w_i\mathsf{R}w'\in\mathcal{B}$, then $w_j\colon\chi;\lfalse$ for all $w_j$'s s.t.\ $w_i\mathsf{R}w_j\in\mathcal{B}$. By the induction hypothesis, we have that $\mathfrak{M},w_j\Vdash^+\chi$ for every $w_j\in R(w_i)$ and, moreover, if $\mathfrak{M},w'\Vdash^-\chi$ for some $w'\in R(w_i)$, then $\mathfrak{M},w_j\Vdash^-\chi$ for all $w_j\in R(w_i)$. Thus, $\mathfrak{M},w_i\Vdash^+\blacksquare\chi$, as required.

For the converse, assume that $w_i\!\colon\!\blacksquare\chi;\ltrue\!\notin\!\mathcal{B}$. Hence, by completeness of $\mathcal{B}$, we have that $w_i\!\colon\!\blacksquare\chi;\lnontrue\!\in\!\mathcal{B}$. Then, one of the following holds:
\begin{enumerate}
\item $w_j\colon\chi;\lnontrue\in\mathcal{B}$ for some $w_j$ s.t.\ $w_i\mathsf{R}w_j\in\mathcal{B}$;
\item $w_j\colon\chi;\lfalse\in\mathcal{B}$ and $w_k\colon\chi;\lnonfalse\in\mathcal{B}$ for some $w_j$ and $w_k$ s.t.\ $w_i\mathsf{R}w_j\in\mathcal{B}$ and $w_i\mathsf{R}w_k\in\mathcal{B}$.
\end{enumerate}
Applying the induction hypothesis, we have that $\mathfrak{M},w_j\nVdash^-\chi$ for some $w_j\in R(w_i)$ in the first case, and $\mathfrak{M},w_j\Vdash^-\chi$ and $\mathfrak{M},w_k\nVdash^-\chi$ for some $w_j,w_k\in R(w_i)$ in the second case. Thus, $\mathfrak{M},w_j\nVdash^+\blacksquare\chi$, as required.

The case when $w_i\colon\blacksquare\chi;\lfalse\in\mathcal{B}$ can be tackled in the same way.

Let us now proceed to the case when $\phi=\mathbf{I}\chi$. Again, we assume that $w_i\colon\mathbf{I}\chi;\ltrue\in\mathcal{B}$. Then, $w_i\chi;\ltrue\in\mathcal{B}$ and for each $w_j$ s.t.\ $w_i\mathsf{R}w_j$\footnote{Observe that $w\mathsf{R}w$, actually, does not occur in $\Tknowigno$ proofs since whenever $w\mathsf{R}w'$ is introduced, $w'$ is fresh.}, $w_j\colon\chi;\lfalse\in\mathcal{B}$. Moreover, if $w'\colon\chi;\ltrue\in\mathcal{B}$ for some $w'$ s.t.\ $w_i\mathsf{R}w'\in\mathcal{B}$, then $w_j\colon\chi;\ltrue\in\mathcal{B}$ for every $w_j$ s.t.\ $w_i\mathsf{R}w_j$. By the induction hypothesis, we obtain that $\mathfrak{M},w_i\Vdash^+\chi$, $\mathfrak{M},w_j\Vdash^-\chi$ for all $w_j\in R^!(w_i)$, and if $\mathfrak{M},w'\Vdash^+\chi$ for some $w'\in R^!(w_i)$, then $\mathfrak{M},w_j\Vdash^+\chi$ for all $w_j\in R^!(w_i)$ as well. Hence, $\mathfrak{M},w_i\Vdash^+\mathbf{I}\chi$.

For the converse, assume that $w_i\colon\mathbf{I}\chi;\ltrue\notin\mathcal{B}$ (hence, $w_i\colon\mathbf{I}\chi;\lnontrue\in\mathcal{B}$). Then, one of the following holds:
\begin{enumerate}
\item $w_i\colon\chi;\lnontrue\in\mathcal{B}$;
\item $w_j\colon\chi;\lnonfalse\in\mathcal{B}$ for some $w_j$ s.t.\ $w_i\mathsf{R}w_j\in\mathcal{B}$;
\item $w_j\colon\chi;\ltrue\in\mathcal{B}$ and $w_k\colon\chi;\lnontrue\in\mathcal{B}$ for some $w_j$ and $w_k$ s.t.\ $w_i\mathsf{R}w_j\in\mathcal{B}$ and $w_i\mathsf{R}w_k\in\mathcal{B}$.
\end{enumerate}
By the induction hypothesis, we have that $\mathfrak{M},w_i\nVdash^+\chi$ in the first case; $\mathfrak{M},w_j\nVdash^-\chi$ for some $w_j\in R^!(w_i)$\footnote{Again, observe that the models produced from the complete open branches of $\Tknowigno$ proofs are \emph{irreflexive}, whence $R^!(w)=R(w)$ for all $w$'s.} in the second case; $\mathfrak{M},w_j\Vdash^+\chi$ and $\mathfrak{M},w_k\nVdash^+\chi$ for some $w_j,w_k\in R^!(w_i)$ in the third case. In all three cases, $\mathfrak{M},w_i\nVdash^+\mathbf{I}\chi$. The case of $w_i\colon\mathbf{I}\chi;\lfalse\in\mathcal{B}$ can be dealt with similarly.
\end{proof}
\section{Proof of Theorem~\ref{theorem:accidenceundefinable}\label{sec:accidenceproof}}
\emph{There is no formula $\phi\in\LBox$ that defines $\bullet p$ on the classes of all frames, all reflexive frames, all transitive frames, all symmetric frames, and all Euclidean frames.}
\begin{proof}
Consider the models in Fig.~\ref{fig:S4counterexample}.
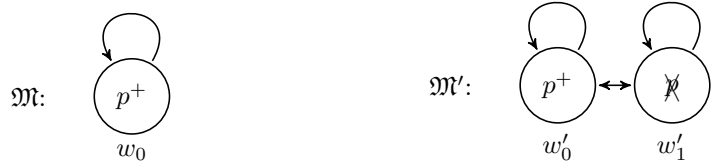
\begin{figure}
\centering
\begin{tikzpicture}[modal,node distance=0.5cm,world/.append style={minimum
size=1cm}]
\node[world] (w0) [label=below:{$w_0$}] {$p^+$};
\node[] [left=of w0] {$\mathfrak{M}$:};
\path[->] (w0) edge[reflexive] (w0);
\end{tikzpicture}
\hfil
\begin{tikzpicture}[modal,node distance=0.5cm,world/.append style={minimum
size=1cm}]
\node[world] (w0) [label=below:{$w'_0$}] {$p^+$};
\node[world] (w1) [right=of w0] [label=below:{$w'_1$}] {$\xcancel{p}$};
\node[] [left=of w0] {$\mathfrak{M}'$:};
\path[->] (w0) edge[reflexive] (w0);
\path[->] (w1) edge[reflexive] (w1);
\path[<->] (w0) edge (w1);
\end{tikzpicture}
\caption{All variables in both models have the same values exemplified by $p$.}
\label{fig:S4counterexample}
\end{figure}
Clearly, $\mathfrak{M},w_0\nVdash^+\bullet p$ and $\mathfrak{M},w_0\Vdash^-\bullet p$ but $\mathfrak{M}',w'_0\Vdash^+\bullet p$ and $\mathfrak{M}',w'_0\nVdash^-\bullet p$. It now suffices to show for any $\phi\in\LBox$ that
\begin{enumerate}
\item if $\mathfrak{M}',w'_0\nVdash^+\phi$ and $\mathfrak{M}',w'_0\Vdash^-\phi$, then $\mathfrak{M},w_0\nVdash^+\phi$ and $\mathfrak{M},w_0\Vdash^-\phi$; and
\item if $\mathfrak{M}',w'_0\Vdash^+\phi$ and $\mathfrak{M}',w'_0\nVdash^-\phi$, then $\mathfrak{M},w_0\Vdash^+\phi$ and $\mathfrak{M},w_0\nVdash^-\phi$.
\end{enumerate}

We proceed by induction on $\phi$. The basis case of variables is trivial since their valuations at $w_0$ and $w'_0$ are the same. We only show the most instructive cases of $\phi=\psi\wedge\psi'$ and $\phi=\lozenge\psi$ (recall that $\lozenge\psi$ can be defined as $\neg\Box\neg\psi$).

\fbox{$\phi=\psi\wedge\psi'$}

For (1), if $\mathfrak{M}',w'_0\nVdash^+\psi\wedge\psi'$ and $\mathfrak{M}',w'_0\Vdash^-\psi\wedge\psi'$, then
\begin{enumerate}
\item[(a)] $\mathfrak{M}',w'_0\nVdash^+\psi$ and $\mathfrak{M}',w'_0\Vdash^-\psi$, or
\item[(b)] $\mathfrak{M}',w'_0\nVdash^+\psi'$ and $\mathfrak{M}',w'_0\Vdash^-\psi'$, or
\item[(c)] w.l.o.g. $\mathfrak{M}',w'_0\Vdash^+\psi$ and $\mathfrak{M}',w'_0\Vdash^-\psi$ but $\mathfrak{M}',w'_0\nVdash^+\psi'$ and $\mathfrak{M}',w'_0\nVdash^-\psi'$.
\end{enumerate}

Cases (a) and (b) hold by the induction hypothesis. For (c), one can show by induction that there is no $\phi\in\LBox$ s.t.\ $\mathfrak{M}',w'_0\Vdash^+\phi$ and $\mathfrak{M}',w'_0\Vdash^-\phi$. The basis case of propositional variables holds by the construction of the model, and the cases of propositional connectives can be shown by a~straightforward application of the induction hypothesis. Finally, one can see that there is no $\chi\in\LBox$ s.t.\ $\mathfrak{M}',w'_1\Vdash^+\chi$ and $\mathfrak{M}',w'_1\Vdash^-\chi$. Thus, it holds that there is no $\phi=\Box\chi$, nor $\phi=\lozenge\chi$ s.t.\ $\mathfrak{M}',w'_0\Vdash^+\phi$ and $\mathfrak{M}',w'_0\Vdash^-\phi$.

For (2), recall that if $\mathfrak{M}',w'_0\Vdash^+\psi\wedge\psi'$ and $\mathfrak{M}',w'_0\nVdash^-\psi\wedge\psi'$, then $\mathfrak{M}',w'_0\Vdash^+\psi$ and $\mathfrak{M}',w'_0\nVdash^-\psi$ as well as $\mathfrak{M}',w'_0\Vdash^+\psi'$ and $\mathfrak{M}',w'_0\nVdash^-\psi'$. Hence, by the induction hypothesis, $\mathfrak{M},w_0\Vdash^+\psi$ and $\mathfrak{M},w_0\nVdash^-\psi$ as well as $\mathfrak{M},w_0\Vdash^+\psi'$ and $\mathfrak{M},w_0\nVdash^-\psi'$. Thus, $\mathfrak{M}',w'_0\Vdash^+\psi\wedge\psi'$ and $\mathfrak{M}',w'_0\nVdash^-\psi\wedge\psi'$, as required.

\fbox{$\phi=\lozenge\psi$}

Observe first, that since $\forall w',w''\!\in\!\mathfrak{M}':w'Rw''$, it holds that
\begin{itemize}
\item $\mathfrak{M}',w'\Vdash^+\lozenge\psi$ iff $\mathfrak{M}',w''\Vdash^+\lozenge\psi$, and
\item $\mathfrak{M}',w'\Vdash^-\lozenge\psi$ iff $\mathfrak{M}',w''\Vdash^-\lozenge\psi$
\end{itemize}
for any $w',w''\in\mathfrak{M}'$ and any $\lozenge\psi\in\LBox$.

Now, consider (1). If $\mathfrak{M}',w'_0\nVdash^+\lozenge\psi$ and $\mathfrak{M}',w'_0\Vdash^-\lozenge\psi$, then $\mathfrak{M}',w'\nVdash^+\psi$ and $\mathfrak{M}',w'\Vdash^-\psi$ for any $w'\in\{w'_0,w'_1\}$. But then, $\mathfrak{M},w_0\nVdash^+\psi$ and $\mathfrak{M},w_0\Vdash^-\psi$ by the induction hypothesis. Hence, $\mathfrak{M},w_0\nVdash^+\lozenge\psi$ and $\mathfrak{M},w_0\Vdash^-\lozenge\psi$, as required.

For (2), let $\mathfrak{M}',w'_0\Vdash^+\lozenge\psi$ and $\mathfrak{M}',w'_0\nVdash^-\lozenge\psi$. We have four options:
\begin{enumerate}
\item[(a)] $\mathfrak{M}',w'_0\Vdash^+\psi$ and $\mathfrak{M}',w'_0\nVdash^-\psi$, or
\item[(b)] $\mathfrak{M}',w'_1\Vdash^+\psi$ and $\mathfrak{M}',w'_1\nVdash^-\psi$, or
\item[(c)] $\mathfrak{M}',w'_1\Vdash^+\psi$, $\mathfrak{M}',w'_1\Vdash^-\psi$, $\mathfrak{M}',w'_0\nVdash^+\psi$ and $\mathfrak{M}',w'_0\nVdash^-\psi$, or
\item[(d)] $\mathfrak{M}',w'_0\Vdash^+\psi$, $\mathfrak{M}',w'_0\Vdash^-\psi$, $\mathfrak{M}',w'_1\nVdash^+\psi$ and $\mathfrak{M}',w'_1\nVdash^-\psi$.
\end{enumerate}

The (a) case holds by the induction hypothesis. Case (d) holds trivially since there is no $\phi\in\LBox$ s.t.\ $\mathfrak{M}',w'_0\Vdash^+\phi$ and $\mathfrak{M}',w'_0\Vdash^-\phi$ as we have just shown. Case (c) holds trivially as well because a~similar inductive argument demonstrates that there is no $\phi\in\LBox$ s.t.\ $\mathfrak{M}',w'_1\Vdash^+\phi$ and $\mathfrak{M}',w'_1\Vdash^-\phi$.

Finally, we reduce (b) to (a) by proving that for any $\psi\in\LBox$, it holds that (i) if $\mathfrak{M}',w'_1\Vdash^+\psi$ and $\mathfrak{M}',w'_1\nVdash^-\psi$, then $\mathfrak{M}',w'_0\Vdash^+\psi$ and $\mathfrak{M}',w'_0\nVdash^-\psi$ as well, and that (ii) if $\mathfrak{M}',w'_1\nVdash^+\psi$ and $\mathfrak{M}',w'_1\Vdash^-\psi$, then $\mathfrak{M}',w'_0\nVdash^+\psi$ and $\mathfrak{M}',w'_0\Vdash^-\psi$. We proceed by induction on $\psi$.

The basis case of a~propositional variable holds trivially. The propositional cases are also straightforward. Now, if $\psi=\lozenge\chi$, we use the observation above to obtain that (i) and (ii) hold as well.

The result follows.
\end{proof}
\end{document}